\documentclass [11pt,reqno]{amsart}
\usepackage {amsmath,amssymb,verbatim,geometry,color,esint}
\usepackage[all]{xy}
\usepackage{mathrsfs}
\usepackage[backref,pagebackref,pdftex,hyperindex]{hyperref}
\usepackage[pdftex]{graphicx}
\usepackage{tikz}
\usetikzlibrary{math}

\usepackage{accents}

\newcommand{\A}{\mathbb{A}}

 \newcommand{\C}{\mathbb{C}}

 \newcommand{\Q}{\mathbb{Q}}
 \newcommand{\R}{\mathbb{R}}
 \newcommand{\Z}{\mathbb{Z}}

\newcommand{\fa}{\mathfrak{a}}

\newcommand{\cC}{\mathcal{C}}
\newcommand{\cD}{\mathcal{D}}
\newcommand{\cE}{\mathcal{E}}

\newcommand{\cL}{\mathcal{L}}
\newcommand{\cM}{\mathcal{M}}

\newcommand{\cO}{\mathcal{O}}

\newcommand{\cZ}{\mathcal{Z}}
\newcommand{\cX}{\mathcal{X}}

\renewcommand{\a}{\alpha}
\renewcommand{\b}{\beta}
\renewcommand{\d}{\delta}
\newcommand{\e}{\varepsilon}
\newcommand{\f}{\varphi}

\newcommand{\g}{\gamma}

\newcommand{\om}{\omega}
\newcommand{\Om}{\Omega}
\newcommand{\Ga}{\Gamma}

\newcommand{\p}{\psi}

\newcommand{\ie}{{\rm i.e.\ }}

\newcommand{\inter}{\cdot\ldots\cdot}
\newcommand{\winter}{\wedge\dots\wedge}

\newcommand{\hto}{\hookrightarrow}

\newcommand{\jj}{\mathrm{J}}
\newcommand{\tJ}{\mathrm{J}^+}

\DeclareMathOperator{\dd}{{d}}

\DeclareMathOperator{\en}{E}

\DeclareMathOperator{\mab}{M}

\DeclareMathOperator{\ii}{I}

\DeclareMathOperator{\env}{P}

\DeclareMathOperator{\Cz}{C^0}

\DeclareMathOperator{\Ent}{Ent}

\DeclareMathOperator{\MA}{MA}

\DeclareMathOperator{\Amp}{Amp}

\DeclareMathOperator{\PSH}{PSH}
\DeclareMathOperator{\CPSH}{CPSH}

\DeclareMathOperator{\Pos}{Pos}

\DeclareMathOperator{\sing}{sing}

\DeclareMathOperator{\effe}{F}

\DeclareMathOperator{\dT}{\dd_{\mathrm{T}}}

\DeclareMathOperator{\Num}{N^1}

\DeclareMathOperator{\VCar}{VCar}

\DeclareMathOperator{\ten}{\widetilde{E}}

\newcommand{\ddc}{\mathrm{dd^c}}

\newcommand{\PH}{\mathcal{PH}}

\renewcommand{\div}{\mathrm{div}}

\newcommand{\triv}{\mathrm{triv}}

\newcommand{\ld}{\mathrm{A}}

\newcommand{\HBC}{\mathrm{H}_{\mathrm{BC}}}

\numberwithin{equation}{section}       

\newtheorem{prop} {Proposition} [section]
\newtheorem{thm}[prop] {Theorem} 
\newtheorem{defi}[prop] {Definition}
\newtheorem{lem}[prop] {Lemma}
\newtheorem{cor}[prop]{Corollary}
\newtheorem{prop-def}[prop]{Proposition-Definition}

\newtheorem*{thmA}{Theorem A} 
 
\newtheorem*{thmB}{Theorem B} 
\newtheorem*{thmC}{Theorem C}

\newtheorem{exam}[prop]{Example}
\newtheorem{rmk}[prop]{Remark}

\newtheorem{conj}[prop]{Conjecture}
\theoremstyle{remark}
\newtheorem*{ackn}{Acknowledgment}

\title[Measures of finite energy]{Measures of finite energy in pluripotential theory: a synthetic approach}
\date{\today}

\author{S{\'e}bastien Boucksom \and Mattias Jonsson}

\address{Sorbonne Universit\'e and Universit\'e Paris Cit\'e\\
CNRS\\
IMJ-PRG\\
F-75005 Paris\\
France}
\email{sebastien.boucksom@imj-prg.fr}

\address{Dept of Mathematics\\
  University of Michigan\\
  Ann Arbor, MI 48109-1043\\
  USA}
\email{mattiasj@umich.edu}

\begin{document}

\begin{abstract} We introduce a synthetic approach to global pluripotential theory, covering in particular the case of a compact K\"ahler manifold and that of a projective Berkovich space over a non-Archimedean field. We define and study the space of measures of finite energy, introduce twisted energy and free energy functionals thereon, and show that coercivity of these functionals is an open condition with respect to the polarization. 
\end{abstract}

\dedicatory{To Bo Berndtsson, on the occasion of his 70+$\e$-th birthday}

\maketitle

\setcounter{tocdepth}{1}
\tableofcontents
%
%
%
%
\section*{Introduction}
Pluripotential theory on compact K\"ahler spaces is by now a very well developed subject, with key applications to K\"ahler geometry. Generalizing classical concepts from potential theory, measures and potentials of finite energy play a central role in this theory, see for instance~\cite{GZ2,BBGZ,thermo,DN,DGL}. In parallel, a non-Archimedean version of pluripotential theory has also emerged, taking place on projective Berkovich spaces~\cite{BerkBook} over a (complete) non-Archimedean field. Initially motivated by Arakelov geometry~\cite{Zha95,CL06}, it has found various other applications, including degenerations of Calabi--Yau manifolds~\cite{LiSYZ} and the Yau--Tian--Donaldson conjecture~\cite{YTD,Li22,nakstab2}. 

These two versions of pluripotential theory bear many similarities, and can be formulated in a quite parallel way.  The main purpose of the present article is to introduce a synthetic formalism covering in particular these two cases, and use it to extend some of the main results of~\cite{trivval} and~\cite{nakstab2} (that were taking place on projective Berkovich spaces over a trivially valued field, and applied to the study of K-stability). More specifically, we 
\begin{itemize}
\item define and study the space of measures of finite energy;
\item introduce the \emph{twisted energy} and \emph{free energy} functionals on the latter space, whose composition with the Monge--Amp\`ere operator respectively recover the Donaldson J-functional and the Mabuchi K-energy in the K\"ahler case, and their analogues in the non-Archimedean case; 
\item show that coercivity of the free energy is an open condition with respect to the K\"ahler class. 
\end{itemize}
The emphasis in this paper is on \emph{measures} of finite energy, as opposed to \emph{potentials} of finite energy, that we do not seek to investigate here (see for instance~\cite{nama,trivval,Reb,DXZ} in the non-Archimedean context). 

%
\subsection*{The general setup}
Throughout this paper, we work with a compact topological space $X$ equipped with the following data: 
\begin{itemize}
\item a dense linear subspace $\cD\subset\Cz(X)$ of \emph{test functions}, containing all constants; 
\item a vector space $\cZ$ of \emph{closed $(1,1)$-forms on $X$}, endowed with a nice\footnote{See~\S\ref{sec:notation} for the terminology.} partial order, and a linear map $\ddc\colon\cD\to\cZ$ vanishing on constants; 
\item an integer $n\ge 1$ (seen as the `dimension' of $X$), and a nonzero $n$-linear symmetric map taking a tuple $(\theta_1,\dots,\theta_n)$ in $\cZ$ to a signed Radon measure $\theta_1\winter\theta_n$ on $X$, assumed to be positive for all $\theta_i\ge 0$, and such that each bilinear form 
\begin{equation}\label{equ:Hodge}
\cD\times\cD\to\R\quad (\f,\p)\mapsto\int\f\,\ddc\p\wedge\theta_1\winter\theta_{n-1}
\end{equation}
with $\theta_i\in\cZ$ is symmetric, and seminegative for $\theta_i\ge 0$. 
\end{itemize}
We then introduce the \emph{Bott--Chern cohomology space} 
$$
\HBC(X):=\cZ/\ddc\cD, 
$$
and define the \emph{positive cone} $\Pos(X)\subset\HBC(X)$ as the interior\footnote{Here we use the finest vector space topology of $\HBC(X)$, see~\S\ref{sec:notation}.} of the image of the convex cone
$$
\cZ_+:=\{\theta\in\cZ\mid\theta\ge 0\}.
$$
This setup is primarily inspired by the case of a compact K\"ahler manifold $X$, where $\cD$ is the space of smooth functions, and $\cZ$ the space of closed $(1,1)$-forms. It also covers the case of a projective Berkovich space $X$ over a complete non-Archimedean field $k$, where $\cD$ is generated by \emph{piecewise linear} (or \emph{model}) functions, and elements of $\cZ$ are represented by numerical classes on models over the valuation ring (or test configurations, in the trivially valued case)~\cite{siminag,GM,trivval}. At least under reasonable assumptions on $X$ and $k$, we then have $\HBC(X)=\Num(X)$, and $\Pos(X)$ coincides with the ample cone, see~\S\ref{sec:BCNA}.  
%
\subsection*{Measures of finite energy}
Fix a form $\om\in\cZ_+$ such that $[\om]\in\Pos(X)$, with volume $V_\om=\int\om^n>0$. The space of \emph{$\om$-plurisubharmonic} test functions is defined as 
$$
\cD_\om:=\{\f\in\cD\mid\om_\f:=\om+\ddc\f\ge 0\}, 
$$
and the \emph{Monge--Amp\`ere operator} takes $\f\in\cD_\om$ to the probability measure 
$$
\MA_\om(\f):=V_\om^{-1}\om_\f^n. 
$$
It admits a primitive, the \emph{Monge--Amp\`ere energy} $\en_\om\colon\cD_\om\to\R$, explicitly given by 
$$
\en_\om(\f)=\frac{1}{n+1}\sum_{j=0}^nV_\om^{-1}\int\f\,\om_\f^j\wedge\om^{n-j}.
$$
Denote by $\cM$ the space of (Radon) probability measures on $X$, and define the \emph{energy} of a measure $\mu\in\cM$ as the Legendre transform\footnote{This corresponds to $\en_\om^\vee(\mu)$ in the notation of~\cite{BBGZ,trivval}, and to $\|\mu\|_\om$ in that of~\cite{nakstab2}.}
$$
\jj_\om(\mu):=\sup_{\f\in\cD_\om}\{\en_\om(\f)-\int\f\,\mu\}\in [0,+\infty]. 
$$
Then $\jj_\om\colon\cM\to [0,+\infty]$ is convex, and lsc in the weak topology; the space of \emph{measures of finite energy} is defined as 
$$
\cM^1_\om:=\{\mu\in\cM\mid\jj_\om(\mu)<\infty\}, 
$$
equipped with the \emph{strong topology}, \ie the coarsest refinement of the weak topology in which $\jj_\om$ becomes continuous. 

As a consequence of the seminegativity of~\eqref{equ:Hodge}, the functional $\en_\om$ is concave. This is equivalent to the non-negativity of the \emph{Dirichlet functional}
$$
\jj_\om(\f,\p):=\en_\om(\f)-\en_\om(\p)+\int(\p-\f)\MA_\om(\f), 
$$
which is more explicitly given by the familiar expression 
$$
\jj_\om(\f,\p)=\frac 12\int(\f-\p)\ddc(\p-\f)
$$
when $n=1$, and a positive linear combination of integrals of the form 
$$
\int(\f-\p)\ddc(\p-\f)\wedge\om_\f^j\wedge\om_\p^{n-j-1}\quad (0\le j<n)
$$
in general, see~\eqref{equ:Dirichlet}. Our first main result shows that the Dirichlet functional induces, via the Monge--Amp\`ere operator, a complete quasi-metric\footnote{See~\S\ref{sec:notation} for the notion of quasi-metric used in this paper.} space structure on $\cM^1_\om$.  

\begin{thmA} Assume that $\om$ has the \emph{orthogonality property}. Then:
\begin{itemize}
\item[(i)] the image of the Monge--Amp\`ere operator $\MA_\om\colon\cD_\om\to\cM$ is a dense subspace of $\cM^1_\om$;
\item[(ii)] there exists a unique (continuous) quasi-metric $\d_\om$ on $\cM^1_\om$ that defines the strong topology of $\cM^1_\om$ and such that 
$$
\d_\om(\MA_\om(\f),\MA_\om(\p))=\jj_\om(\f,\p)
$$
for all $\f,\p\in\cD_\om$; 
\item[(iii)] the quasi-metric space $(\cM^1_\om,\d_\om)$ is complete.
\end{itemize}
\end{thmA}
The energy can be expressed in terms of the quasi-metric as 
$$
\jj_\om(\mu)=\d_\om(\mu,\mu_\om)\quad\text{where}\quad\mu_\om:=V_\om^{-1}\om^n=\MA_\om(0). 
$$
We refer to Definition~\ref{defi:ortho} for the precise definition of the orthogonality property. Suffice it to say here that it only depends on $[\om]\in\Pos(X)$, and holds when $X$ is a compact K\"ahler manifold, or any projective Berkovich space $X$ over a non-Archimedean field (as a consequence of~\cite{BE,BGM}). In the former case, Theorem~A can be deduced from~\cite{BBGZ,BBEGZ}; in the latter, it was established in the trivially valued case in~\cite{trivval}, and is thus extended here to the case of an arbitrary non-Archimedean ground field. 

The strategy of proof of Theorem~A follows the same lines as the trivially case treated in~\cite{trivval}. The first key ingredient is a uniform differentiability property for the Legendre transform of the convex functional $\mu\mapsto\jj_\om(\mu)$, which is shown to be equivalent to the orthogonality property. This is used to prove that if $(\f_i)$ is a maximizing sequence for a given $\mu\in\cM^1_\om$ (\ie a sequence in $\cD_\om$ computing the supremum that defines $\jj_\om(\mu)$), then $\MA_\om(\f_i)$ converges to $\mu$. 

The rest of the proof relies on an extensive use of H\"older estimates for mixed Monge--Amp\`ere integrals, obtained from repeated applications of the Cauchy--Schwarz inequality to the seminegative form~\eqref{equ:Hodge}. This approach, which goes back to~\cite{Blo} and was further exploited in~\cite{BBGZ,BBEGZ, trivval}, is put in a simple general setting in Appendix~\ref{sec:CS}. 

%
\subsection*{Twisted energy, free energy, and coercivity} 
Assuming from now on the orthogonality property, we next investigate the dependence on $\om$ of the space $\cM^1_\om$ and the energy functional $\jj_\om\colon\cM^1_\om\to\R_{\ge 0}$. To this end, we require the \emph{submean value property}, \ie the existence of $C\in\R_{\ge 0}$ such that 
$$
\sup \f\le \int\f\,\mu_\om+C
$$ 
for all $\om$-psh test functions $\f\in\cD_\om$. We show that this condition is independent of $\om$, and that it is equivalent to the irreducibility of $X$ when the latter is a compact K\"ahler or projective Berkovich space (see~\S\ref{sec:submean}). 

\begin{thmB} Assume that the submean value property holds. Then: 
\begin{itemize}
\item the topological space $\cM^1=\cM^1_\om$ is independent of $\om$; 
\item for any $\theta\in\cZ$, there exists a unique continuous functional $\jj_\om^\theta\colon\cM^1\to\R$ such that 
$$
\jj_\om^\theta(\mu)=\frac{d}{dt}\bigg|_{t=0}\jj_{\om+t\theta}(\mu)
$$
for any $\mu\in\cM^1$; furthermore, $\jj_\om^\theta(\mu)$ satisfies H\"older estimates with respect to $\om$. 
\end{itemize}
\end{thmB}

The strategy of proof of Theorem~B again globally follows the same lines as the trivially valued case treated in~\cite{trivval,nakstab2}. However, in the latter case the submean value inequality is actually an equality, \ie one can take the constant $C$ above to be $0$, while an extra layer of complication arises in the general case to handle this constant, which forces us to take a slightly different route. 

We call $\jj_\om^\theta(\mu)$ the \emph{$\theta$-twisted energy} of $\mu\in\cM^1$. It provides an analogue of the \emph{Donaldson J-functional} on the level of measures, and its relevance comes from its relation to the \emph{Mabuchi K-energy}, when $X$ is a compact K\"ahler manifold or a smooth projective Berkovich space. In these two cases, the choice of a (smooth or PL) metric $\rho$ on the canonical bundle $K_X$ defines an \emph{entropy} functional $\Ent\colon\cM\to\R\cup\{+\infty\}$. In line with~\cite{thermo}, we then define the \emph{free energy} $\effe_\om\colon\cM^1\to\R\cup\{+\infty\}$ by setting
$$
\effe_\om(\mu):=\Ent(\mu)-\Ent(\mu_\om)+\jj_\om^\theta(\mu),
$$
where $\theta\in\cZ$ denotes the curvature of $\rho$. The free energy so defined is independent of the choice of $\rho$, and its composition with the Monge--Amp\`ere operator coincides with the \emph{Mabuchi K-energy} $\mab_\om\colon\cD_\om\to\R$. As a consequence of Theorem~B, we then show:

\begin{thmC} Assume $X$ is a compact K\"ahler manifold or smooth projective Berkovich space over a non-Archimedean field, and consider the \emph{coercivity threshold} 
$$
\sigma(X,\om):=\sup\left\{\sigma\in\R\mid\effe_\om\ge\sigma\jj_\om+A\text{ for some }A\in\R\right\}. 
$$
Then $\sigma(X,\om)$ is a continuous function of $[\om]\in\Pos(X)$. 
\end{thmC}
This result actually holds in much greater generality, for the twisted coercivity threshold of any given functional on $\cM^1$ (see Theorem~\ref{thm:threshcont}). 

In the trivially valued case, the free energy $\effe_\om(\mu)$ coincides with the invariant $\b_\om(\mu)$ defined and studied in~\cite{nakstab2}, and $\sigma(X,\om)$ with the \emph{divisorial stability threshold} of $(X,\om)$, which is positive iff $(X,\om)$ is \emph{divisorially stable} (a strengthening of uniform K-stability, conjecturally equivalent to it, cf.~\S\ref{sec:coerNA}). Specializing to the case of Dirac measures $\mu$ recovers the notion of \emph{valuative stability}~\cite{DL,LiuYa}, which in the Fano case is equivalent to K-stability~\cite{Fujval,LiEquivariant}.

In the case of a compact K\"ahler manifold, we have $\sigma(X,\om)>0$ iff $[\om]$ contains a unique constant scalar curvature K\"ahler (cscK) metric, as a consequence of~\cite{CC} and~\cite{DaR,BDL20}. Theorem~C thus recovers the fact, originally due to LeBrun-Simanca~\cite{LS}, that the existence of a unique cscK metric in a K\"ahler class is an open condition on that class.

%
%
\subsection*{Structure of the paper}
The article is organized as follows. 

\begin{itemize}
\item Section~\ref{sec:synthetic} introduces the synthetic pluripotential theoretic formalism, including the energy pairing and the submean value property, and establishes basic properties of the Dirichlet functional.

\item Section~\ref{sec:mes} studies the space of measures of finite energy. It introduces the orthogonality property, and proves Theorem~A (cf.~\S\ref{sec:Dirqm}). 

\item Assuming the submean value property, Section~\ref{sec:further} establishes the first part of Theorem~B, along with various further estimates for the energy. 

\item Section~\ref{sec:twisteddiff} is devoted to the twisted energy, which is proved to compute the directional derivatives of the energy, concluding the proof of Theorem~B. 

\item Section~\ref{sec:freeen} studies the (twisted) coercivity threshold of a functional, and proves that it depends continuously on the cohomology classes. This is then applied to the free energy, yielding Theorem~C. 

\item Finally, Appendix~A establishes the relevant estimates needed for the Dirichlet functional in a simple general setting, while Appendix~B studies the orthogonality property on compact K\"ahler spaces.  

\end{itemize}

%
%
\begin{ackn}
The authors would like to thank Robert Berman, Ruadha\'i Dervan, Philippe Eyssidieux, Charles Favre, Vincent Guedj, L\'eonard Pille-Schneider, R\'emi Reboulet and Ahmed Zeriahi for helpful discussions related to the contents of this paper. 

It is a great pleasure to dedicate this article to Bo Berndtsson. His many fundamental contributions to complex analysis, pluripotential theory and convex analysis have been a constant source of inspiration for us. We would especially like to express our admiration for his fantastic results on positivity of direct images that have played a key role in our own works (among many others!), as well as for his laid-back approach to mathematics in general, which makes Bo such a pleasant person to interact with. 

The second author was partially supported by NSF grants DMS-1900025 and DMS-2154380.
\end{ackn}

%
%
\section{A synthetic pluripotential theoretic formalism}\label{sec:synthetic}
In this section, we introduce the general setup considered in this paper, which is designed to cover in a synthetic manner the case of a compact K\"ahler space and that of a projective Berkovich space over a non-Archimedean field. 
%
%
%
\subsection{Notation and terminology}\label{sec:notation}

\begin{itemize}
\item For $x,y\in\R$, $x\lesssim y$ or $x=O(y)$ mean in this paper $x\le C_n  y$ for a constant $C_n>0$ only depending on a given integer $n$ fixed in the setup, and $x\approx y$ if $x\lesssim y$ and $y\lesssim x$. 

\item Recall that any real vector space $V$ admits a \emph{finest vector space topology}, generated by its finite dimensional subspaces, \ie a subset $A\subset V$ is open (or closed) iff, for each finite dimensional subspace $W\subset V$, $A\cap W$ is open (resp.~closed) in the canonical vector space topology of $W$. This topology is not locally convex as soon as the dimension of $V$ is uncountable. 

\item Consider a partially ordered $\R$-vector space $(V,\ge)$. We shall say for brevity that the partial order is \emph{nice} if $V_+:=\{x\in V\mid x\ge 0\}$ spans $V$, and is closed in the finest vector space topology of $V$. 

\item In this paper, a \emph{quasi-metric} on a set $Z$ is a function $\d\colon Z\times Z\to\R_{\ge 0}$ that
\begin{itemize}
\item is \emph{quasi-symmetric}, \ie there exists $C>0$ such that 
$$
C^{-1}\d(x,y)\le \d(y,x)\le C \d(x,y); 
$$
\item satisfies the \emph{quasi-triangle inequality}, \ie there exists $C>0$ such that 
$$
\d(x,y)\le C(\d(x,z)+\d(z,y))
$$
(or, equivalently, $\d(x,y)\le C'\max\{\d(x,z),\d(z,y)\}$ for some other constant $C'>0$); 
\item \emph{separates points}, \ie $\d(x,y)=0\Leftrightarrow x=y$.
\end{itemize}
Note that $\d^\e$ is then also a quasi-metric for any $\e>0$. A quasi-metric space $(Z,\d)$ comes with a Hausdorff topology, and even a uniform structure. In particular, Cauchy sequences and completeness make sense for $(Z,\d)$. Such uniform structures have a countable basis of entourages, and are thus metrizable, by general theory. In fact, there exists a metric $d$ on $X$ and $A>0$ and constants $\e>0$ only depending on $C$ such that the quasi-metric $\d^\e$ satisfies $A^{-1}\d^\e\le d\le A\d^\e$ (see~\cite{Fri,AIN,PS}, thanks to Prakhar Gupta for hinting at those references). 
\end{itemize}

%
%
\subsection{Test functions and closed $(1,1)$-forms}\label{sec:setup}
Throughout this paper, we work with a compact Hausdorff topological space $X$. We denote by $\Cz(X)^\vee$ the space of signed Radon measures on $X$, and by 
$$
\cM\subset\Cz(X)^\vee
$$
the subset of probability measures, which is convex and compact in the weak topology. Recall from the introduction that we assume $X$ equipped with the following data: 
\begin{itemize}
\item a dense linear subspace $\cD\subset\Cz(X)$ of \emph{test functions}, containing all constants; 
\item a vector space $\cZ$ of \emph{closed $(1,1)$-forms on $X$}, endowed with a nice partial order, and a linear map $\ddc\colon\cD\to\cZ$ vanishing on constants; 
\item an integer $n\ge 1$ (viewed as the `dimension' of $X$), and a nonzero $n$-linear symmetric map taking a tuple $(\theta_1,\dots,\theta_n)$ in $\cZ$ to a signed Radon measure $\theta_1\winter\theta_n$ on $X$, assumed to be positive for all $\theta_i\ge 0$, and such that each bilinear form 
$$
\cD\times\cD\to\R\quad (\f,\p)\mapsto\int\f\,\ddc\p\wedge\theta_1\winter\theta_{n-1}
$$
with $\theta_i\in\cZ$ is symmetric, and seminegative when $\theta_i\ge 0$. 
\end{itemize}
Since $\cZ_+$ generates $\cZ$, symmetry in the last item amounts to the \emph{integration-by-parts} formula
\begin{equation}\label{equ:intpart}
\int\f\,\ddc\p\wedge\theta_1\winter\theta_{n-1}=\int\p\,\ddc\f\wedge\theta_1\winter\theta_{n-1}
\end{equation}
for all $\f,\p\in\cD$ and $\theta_i\in\cZ$, while seminegativity requires
\begin{equation}\label{equ:semineg}
\int\f\,\ddc\f\wedge\theta_1\winter\theta_{n-1}\le 0
\end{equation}
when $\theta_i\ge 0$ for all $i$. 

%
%
%

\begin{defi} For any $\theta\in\cZ$ and $\f\in\cD$, we set $\theta_\f:=\theta+\ddc\f$. We say that $\f$ is \emph{$\theta$-plurisubharmonic} ($\theta$-psh for short) if $\theta_\f\ge 0$, and denote by 
$$
\cD_\theta:=\{\f\in\cD\mid\theta_\f\ge 0\}
$$
the space of $\theta$-psh test functions. 
\end{defi}
For all $\theta,\theta'\in\cZ$ and $t\in\R_{>0}$ we have
$$
\cD_\theta+\cD_{\theta'}\subset\cD_{\theta+\theta'},\quad\cD_{t\theta}=t\cD_\theta.
$$
In particular, $\cD_\theta$ is a convex subset of $\cD$, and 
\begin{equation}\label{equ:Dmon}
\theta\le\theta'\Longrightarrow\cD_\theta\subset\cD_{\theta'}.
\end{equation}
Since $\ddc$ vanishes on constants, we have: 
\begin{exam}\label{exam:cstpsh} Constant functions on $X$ are $\theta$-psh iff $\theta\ge 0$. 
\end{exam}

The main two instances of the above setup considered in this paper are as follows. 

\subsubsection{The K\"ahler case}\label{sec:Kahler} The above formalism is primarily inspired by the case of a compact K\"ahler complex analytic space $X$. Here $n=\dim X$, $\cD=\cC^\infty(X)$ is the space of smooth functions on $X$, and $\cZ$ is the space of closed $(1,1)$-forms on $X$ that are locally $\ddc$-exact, \ie global sections of the image $\cZ^{1,1}_X$ of the sheaf morphism $\ddc\colon\cC^\infty_X\to\Omega^{1,1}_X$ (see for instance~\cite{DemSMF} for the definition of smooth forms in this context). When $X$ is nonsingular, \ie a compact K\"ahler manifold, $\cZ$ coincides with the space of closed $(1,1)$-forms on $X$, but the inclusion might be strict in the singular case. 

For all $\f,\p\in\cD$ and $\theta_i\in\cZ$, the Stokes formula implies 
$$
\int\f\,\ddc\p\wedge\Theta=\int\p\,\ddc\f\wedge\Theta=-\int \mathrm{d}\f\wedge\mathrm{d^c}\p\wedge\Theta
$$
with $\Theta:=\theta_1\winter\theta_{n-1}$, which yields~\eqref{equ:intpart} and~\eqref{equ:semineg}, since the $(1,1)$-form $\mathrm{d}\f\wedge\mathrm{d^c}\f$ is semipositive. 

\begin{rmk}\label{rmk:npos} More generally, fix a K\"ahler form $\om_X$ on $X$ and assume $n\le \dim X$. Recall that a $(1,1)$-form $\a$ on $X$ is \emph{$n$-semipositive} (with respect to $\om_X$) if $\a^i\wedge\om_X^{\dim X-i}\ge 0$ for $i=0,\dots,n$ (see~\cite{Blo2,LN}). For any $n$-tuple of $n$-semipositive $(1,1)$-forms $\a_1,\dots,\a_n$, we have $\a_1\winter\a_n\wedge\om_X^{\dim X-n}\ge 0$. One thus obtains a setup satisfying the above conditions by defining $\cZ_+$ as the convex cone of closed $(1,1)$-forms that are $n$-semipositive, and by associating to a tuple $(\theta_1,\dots,\theta_n)$ in $\cZ$ the measure $\theta_1\winter\theta_n\wedge\om_X^{\dim X-n}$.
\end{rmk}

\subsubsection{The non-Archimedean case}\label{sec:NA} Assume now that $X$ is a projective Berkovich space over a non-Archimedean field $k$, \ie the Berkovich analytification of a projective $k$-scheme, of dimension $n=\dim X$. We then take $\cD$ to be the $\R$-vector space generated by \emph{PL functions}, see~\cite[\S 5.4]{BE}. 

When $k$ is nontrivially valued, $\cD$ can be described in terms of \emph{vertical divisors} on (projective, flat) models $\cX$ of $X$ over the spectrum $S$ of the valuation ring. More precisely, we have
$$
\cD\simeq\varinjlim_\cX\VCar(\cX)_\R,
$$
where $\VCar(\cX)_\R$ denotes the $\R$-vector space generated by Cartier divisors on $\cX$ that are vertical, \ie supported on the special fiber. The same description applies in the trivially valued case as well, if a model is now understood as a \emph{test configuration} $\cX\to S:=\A^1$ (see~\cite[\S 6.1]{BHJ1}, \cite[\S 2.2]{trivval}). 

In the nontrivially valued case, the space $\cZ$ is defined by setting
\begin{equation}\label{equ:ZNA}
\cZ:=\varinjlim_\cX\Num(\cX/S),
\end{equation}
where $\Num(\cX/S)$ denotes the (finite dimensional) vector space of relative numerical classes (see~\cite[\S 4.2]{siminag}, \cite[\S 4]{GM}, the definition being inspired by~\cite{BGS}). A form $\theta\in\cZ$ is thus represented by a numerical class $\theta_\cX\in\Num(\cX/S)$ for some model $\cX$, called a \emph{determination} of $\theta$, two such classes being identified if they coincide after pulling back to some higher model, and we write $\theta\ge 0$ if $\theta_\cX$ is (relatively) nef for some (hence any) determination $\cX$ of $\theta$. 

The measure $\theta_1\winter\theta_n$ associated to a tuple of forms $\theta_i\in\cZ$ is a finite linear combination of Dirac masses at divisorial points, whose coefficients can be described in terms of intersection numbers computed on models. 

The linear map $\ddc\colon\cD\to\cZ$ takes a vertical divisor $D\in\VCar(\cX)_\R$ to its numerical class in $\Num(\cX/S)$, and the seminegativity condition~\eqref{equ:semineg} follows the local Hodge index theorem of Yuan--Zhang~\cite[Theorem~2.1]{YZ}. 

Again, the same discussion applies to the trivially valued case as well, using test configurations instead of models. In that case, pulling back numerical classes on $X$ to the product test configuration further yields an injection
\begin{equation}\label{equ:numtriv}
\Num(X)\hto\cZ.
\end{equation}
Note that only forms lying in $\Num(X)$ were considered in~\cite{trivval}. 

\begin{rmk} In the non-Archimedean case, one could also work with smooth functions and $(1,1)$-forms in the sense of~\cite{CLD} (see also~\cite{GK15,GK17,GJR}), but this will not be considered in the present paper. 
\end{rmk}

\begin{rmk} More generally, test configurations for an arbitrary compact K\"ahler manifold (as in~\cite{SD,DeR}) can also be approached using the above formalism. This is the topic of forthcoming work of Pietro Piccione. 
\end{rmk}

\begin{rmk} Another setting where the above formalism applies is that of tropical toric pluripotential theory as in~\cite{BGJK}.
\end{rmk}

%
%
\subsection{Bott--Chern cohomology and positive classes}\label{sec:pos}

\begin{defi} We define the \emph{Bott--Chern cohomology space} as 
$$
\HBC(X):=\cZ/\ddc\cD,  
$$
and denote by $\theta\mapsto[\theta]$ the quotient map $\cZ\to\HBC(X)$. The \emph{positive cone} 
$$
\Pos(X)\subset\HBC(X)
$$
is defined as the interior of the image of $\cZ_+$. 
\end{defi}
Here the interior is taken with respect to the finest vector space topology (see~\S\ref{sec:notation}). Concretely, a class $\a\in\HBC(X)$ belongs to $\Pos(X)$ if, for any $\b\in\HBC(X)$, $\a+t\b$ lies in the image of $\cZ_+$ for all $t\in\R$ small enough.  

Since it is assumed that $\cZ_+$ spans $\cZ$, its image in $\HBC(X)$ is a convex cone that generates $\HBC(X)$. As a consequence, the positive cone $\Pos(X)$ is non-empty as soon as $\HBC(X)$ is finite dimensional. 

\begin{rmk} The image $\cZ_+$ in $\HBC(X)$ is not closed in general. This happens already in the complex projective case, see~\cite[Example~1.7]{DPS}.
\end{rmk}

Since $\ddc$ vanishes on $\R\subset\cD$, \eqref{equ:intpart} yields 
$$
\int\ddc\f\wedge\theta_1\winter\theta_{n-1}=\int\f\,\ddc 1\wedge\theta_1\winter\theta_{n-1}=0
$$
for all $\f\in\cD$ and $\theta_i\in\cZ$. As a result, $(\theta_1,\dots,\theta_n)\mapsto\int\theta_1\winter\theta_n$ descends to a symmetric $n$-linear pairing 
$$
\HBC(X)^n\to\R\quad (\a_1,\dots,\a_n)\mapsto\a_1\inter\a_n,
$$
which we call the \emph{intersection pairing}. 

\begin{lem}\label{lem:posint} For all classes $\a_1,\dots,\a_n\in\Pos(X)$ we have $\a_1\inter\a_n>0$.
\end{lem}
\begin{proof} By assumption, the measure $\theta_1\winter\theta_n$ is nonzero for some tuple $\theta_1,\dots,\theta_n\in\cZ$. Since $\cZ_+$ generates $\cZ$, we can further assume $\theta_i\in\cZ_+$ for all $i$. Then $[\theta_1]\inter [\theta_n]=\int\theta_1\winter\theta_n>0$. Now we can find $0<\e\ll 1$ such that $\a_i-\e [\theta_i]\in\Pos(X)$ for all $i$, and hence $\a_1\inter\a_n\ge\e^n[\theta_1]\inter[\theta_n]>0$. 
\end{proof}

\begin{defi} For each $\om\in\cZ_+$ and $\theta\in\cZ$ we set
$$
\|\theta\|_\om:=\inf\{C\ge 0\mid \pm\theta\le C\om\}\in [0,+\infty],
$$
and we say that $\theta$ is \emph{$\om$-bounded} if $\|\theta\|_\om<\infty$. 
\end{defi} 
The set
$$
\cZ_\om:=\{\theta\in\cZ\mid\|\theta\|_\om<\infty\}
$$ 
of $\om$-bounded forms is a linear subspace of $\cZ$, on which $\|\cdot\|_\om$ defines a norm. 

\begin{rmk} In general, $\cZ_\om$ is a strict subspace of $\cZ$. More precisely, $\cZ_\om=\cZ$ iff $\om$ lies in the interior of $\cZ_+$ in the finest vector space topology of $\cZ$, which is empty in the non-Archimedean case (see \S\ref{sec:BCNA} below). 
\end{rmk}

We can now characterize positive classes as follows: 

\begin{prop}\label{prop:normfin} A class $\a\in\HBC(X)$ lies in $\Pos(X)$ iff, for each finite dimensional subspace $V\subset\cZ$, $\a$ admits a representative $\om\in\cZ_+$ such that $V\subset\cZ_\om$.
\end{prop}
\begin{proof} Assume $\a\in\Pos(X)$, and pick a basis $(\theta_i)_{1\le i\le r}$ of a finite dimensional vector space $V\subset\cZ$. Since $\cZ_+$ spans $\cZ$, for each $i$ we can write $\theta_i=\theta_i^+-\theta_i^-$ with $\theta_i^\pm\ge 0$. Since $\a\in\Pos(X)$ we have $\a-\e[\theta_i^\pm]\in\Pos(X)$ for all $0<\e\ll 1$, and we can thus find $\e>0$ and $\om_i^\pm\in\a$ such that $\om_i^\pm-\e\theta_i^\pm\ge 0$ for all $i$. Now set 
$$
\om:=\tfrac 1{2r}\sum_{i=1}^r(\om_i^++\om_i^-)\in\cZ_+.
$$
Then $[\om]=\a$, and for each $i$ we have $\om\ge\tfrac \e{2r}\theta_i^\pm\ge\pm\tfrac\e{2r}\theta_i$, and hence $\theta_i\in\cZ_\om$, \ie $V\subset\cZ_\om$. This proves the `only if' part, and the converse is clear. 
\end{proof}

\begin{cor}\label{cor:pshspan} Pick $\theta\in\cZ$ such that $[\theta]\in\Pos(X)$. Then any $f\in\cD$ can be written as $f=f^+-f^-$ with $f^\pm\in\cD_{t\theta}=t\cD_\theta$ for some $t>0$. In particular, $\cD_\theta$ spans $\cD$. 
\end{cor}
\begin{proof} By Proposition~\ref{prop:normfin}, we can find $\p\in\cD_\theta$ and $C>0$ such that $-\ddc f\le t(\theta+\ddc\p)$. Thus $f^+:=\f+t\p$ lies in $\cD_{t\theta}$, and the result follows with $f^-:=t\p\in\cD_{C\theta}$. 
\end{proof}

Following~\cite{Tho}, we define the \emph{Thompson distance} between $\om,\om'\in\cZ_+$ as
\begin{equation}\label{equ:dT}
\dT(\om,\om'):=\inf\{\d\ge 0\mid e^{-\d}\om\le\om'\le e^\d\om\}\in [0,+\infty]. 
\end{equation}
We say that $\om$ and $\om'$ are \emph{commensurable} if $\dT(\om,\om')<\infty$. Note that this holds iff $\om'\in\cZ_\om$ and $\om\in\cZ_{\om'}$. Commensurability is an equivalence relation on $\cZ_+$. The linear subspace $\cZ_\om$ only depends on the commensurability class of $\om\in\cZ_+$, and so does the equivalence class of the norm $\|\cdot\|_\om$. 

The next result is readily checked, and left to the reader. 

\begin{lem}\label{lem:commdT} The commensurability class of any $\om\in\cZ_+$ forms an open convex cone in the normed vector space $(\cZ_\om,\|\cdot\|_\om)$, whose topology is defined by the restriction of the Thompson metric. 
\end{lem}

We conclude this section with the following useful fact, that will be put to use in~\S\ref{sec:further}. 

\begin{prop}\label{prop:comm} Each finite subset of $\Pos(X)$ can be represented by commensurable forms in $\cZ_+$. 
\end{prop}
\begin{proof} We first claim that any two $\a,\a'\in\Pos(X)$ admit commensurable representatives. Indeed, setting $\a\sim\a'$ when $\a$ and $\a'$ admit commensurable representatives defines an equivalence relation on $\Pos(X)$, and it is thus enough that each equivalence class intersects the line segment $[\a,\a']$ in an open set. Pick $\b\in [\a,\a']$, and choose representatives $\theta,\theta'\in\cZ_+$ of $\a,\a'$. By Proposition~\ref{prop:normfin}, we can find a representative $\om\in\cZ_+$ of $\b$ such that $\theta,\theta'\in\cZ_\om$. By Lemma~\ref{lem:commdT}, the commensurability class of $\om$ intersects $[\theta,\om]$ and $[\om,\theta']$ along neighborhood of $\om$, and it follows, as desired, that the equivalence class of $\b$ intersects $[\a,\a']$ along a neighborhood of $\b$. 

Now pick $\a_1,\dots,\a_r\in\Pos(X)$. By the first part of the proof, for each $i=2,\dots,r$ we can find two commensurable forms $\om_{1,i},\om_i\in\cZ_+$ such that $[\om_{1,i}]=\a_1$ and $[\om_i]=\a_i$. Then $\om_1:=\tfrac{1}{r-1}\sum_{i=2}^r\om_{1,i}\in\cZ_+$ represents $\a_1$, and is commensurable to $\om_i$ for all $i$ since commensurability classes are convex cones.
\end{proof}

 \subsubsection{The K\"ahler case}\label{exam:BCKahler} As in~\S\ref{sec:Kahler}, assume that $X$ is a compact K\"ahler space. When $X$ is nonsingular, $\HBC(X)$ coincides with the $(1,1)$-part of the Hodge decomposition of $\mathrm{H}^2(X,\R)$. In general, denote by $\PH_X$ the sheaf of germs of \emph{pluriharmonic functions}, \ie the kernel of the sheaf morphism $\ddc\colon\cC^\infty_X\to\Om^{1,1}_X$. The existence of partitions of unity implies that the sheaf $\cC^\infty_X$ is soft, and the cohomology long exact sequence associated to the short exact sequence of sheaves
$$
0\to\PH_X\to\cC^\infty_X\to\cZ^{1,1}_X\to 0
$$
thus yields 
$$
\HBC(X)\simeq\mathrm{H}^1(X,\PH_X).
$$
We do not know whether this is always finite dimensional, but this holds at least when $X$ is normal, as a consequence of the fact that $\PH_X$ then coincides with the sheaf $\Re\cO_X$ of real parts of holomorphic functions, see~\cite[\S 4.6.1]{BG}. 

The positive cone $\Pos(X)$ is compatible with the usual definition, \ie $\a\in\HBC(X)$ lies in $\Pos(X)$ iff $\a$ can be represented by a K\"ahler form. Indeed, given any K\"ahler form $\om$, a class $\a\in\Pos(X)$ can be represented by a form $\theta\in\cZ$ such that $\theta-\e\om\in\cZ_+$ for some $0<\e<\ll 1$, so that $\theta$ is a K\"ahler form. In particular, K\"ahler forms constitute a single commensurability class that maps onto $\Pos(X)$. 

\subsubsection{The non-Archimedean case}\label{sec:BCNA}  Assume, as \S\ref{sec:NA}, that $X$ is a projective Berkovich space over a non-Archimedean field. By definition, any $\theta\in\cZ$ is represented by a numerical class on some model $\cX$ of $X$, whose restriction to the generic fiber defines a numerical class $\{\theta\}\in\Num(X)$ on (the projective variety underlying) $X$. This induces a surjective map 
\begin{equation}\label{equ:HBCNA}
\HBC(X)\twoheadrightarrow\Num(X),
\end{equation}
which is an isomorphism when $X$ is smooth and $k$ is discretely valued of residue characteristic $0$~\cite[Theorem~4.3]{siminag}, or when $X$ is normal and $k$ is algebraically closed~\cite[Theorem~4.2.7]{Jell}). It is also an isomorphism when $k$ is trivially valued, with inverse provided by pulling back classes in $\Num(X)$ to the trivial test configuration $\cX_\triv=X\times\A^1$. 

We claim that $\Pos(X)$ coincides with the preimage of the ample cone of $X$ under~\eqref{equ:HBCNA}, and that $\cZ_+$ has empty interior in the finest vector space topology of $\cZ$ (so that K\"ahler forms do not admit an analogue in the non-Archimedean case). 

To see this, we say as  in~\cite[\S 5.1]{siminag} that a form $\om\in\cZ_+$ is \emph{$\cX$-positive} for a given model/test configuration $\cX$ if it is represented by a (relatively) ample class in $\Num(\cX/S)$. Note that all $\cX$-positive forms are commensurable. As observed in~\cite[Proposition~5.2]{siminag} (see also~\cite[Proposition~4.14]{GM}, \cite[Lemma~3.11]{trivval}), any $\a\in\HBC(X)$ whose image in $\Num(X)$ under~\eqref{equ:HBCNA} is ample can be represented by an $\cX$-positive form $\om$ for any sufficiently high model $\cX$. Now $\cZ_\om$ contains all forms determined on $\cX$, but does not contain any $\cX'$-positive form for a model $\cX'$ strictly dominating $\cX$. The claim easily follows.

%
%
%
\subsection{The energy pairing}\label{sec:enpair}
Mimicking the properties of induced metrics on Deligne pairings (see for instance~\cite[Theorem~8.16]{BE}), and generalizing~\cite[\S 2.2]{SD} (for a compact K\"ahler manifold) and~\cite[\S 3.2]{trivval} (for a projective Berkovich space over a trivially valued field), we introduce: 

\begin{prop-def}\label{propdef:en} The \emph{energy pairing} 
$$
(\cZ\times\cD)^{n+1}\to\R\quad((\theta_0,\f_0),\dots,(\theta_n,\f_n))\mapsto(\theta_0,\f_0)\inter(\theta_n,\f_n),
$$
is defined as the unique $(n+1)$-linear symmetric map such that 
\begin{equation}\label{equ:enint}
(0,\f_0)\cdot(\theta_1,\f_1)\inter(\theta_n,\f_n)=\int\f_0\,(\theta_1+\ddc\f_1)\winter(\theta_n+\ddc\f_n)
\end{equation} 
and
\begin{equation}\label{equ:ennorm}
(\theta_0,0)\inter(\theta_n,0)=0. 
\end{equation}
It further satisfies 
\begin{equation}\label{equ:encst}
(\theta_0,\f_0+c_0)\inter(\theta_n,\f_n+c_n)=(\theta_0,\f_0)\inter(\theta_n,\f_n)+\sum_{i=0}^n c_i[\theta_0]\inter\widehat{[\theta_i]}\inter[\theta_n]
\end{equation}
for all $c_i\in\R$, and 
\begin{equation}\label{equ:pairingtrans} 
\left(\theta_0+\ddc\tau_0,\f_0\right)\inter\left(\theta_n+\ddc\tau_n,\f_n\right)=(\theta_0,\tau_0+\f_0)\inter(\theta_n,\tau_n+\f_n)-(\theta_0,\tau_0)\inter(\theta_n,\tau_n). 
\end{equation}
for all $\tau_i\in\cD$. 
\end{prop-def}

\begin{proof} Using multilinearity, symmetry and~\eqref{equ:enint}, \eqref{equ:ennorm}, we necessarily have 
\begin{equation}\label{equ:enpairing}
(\theta_0,\f_0)\inter(\theta_n,\f_n)=\sum_{i=0}^n\int\f_i\,\theta_0\winter\theta_{i-1}\wedge(\theta_{i+1}+\ddc\f_{i+1})\winter(\theta_n+\ddc\f_n).
\end{equation}
This proves uniqueness. To show existence, the only nontrivial part is to check that the right-hand side of~\eqref{equ:enpairing} is a symmetric function of the $(\theta_i,\f_i)$. It suffices to see invariance under transpositions, which is an easy consequence of the integration-by-parts formula~\eqref{equ:intpart} (compare~\cite[Proposition~2.3]{SD}). 

To see~\eqref{equ:encst}, we may assume $c_i=0$ for $i>0$, and the result is then a direct consequence of~\eqref{equ:enint}. 

Finally, pick $\tau_i\in\cD$, and set
$$
F(\f_0,\dots,\f_n):=(\theta_0,\tau_0+\f_0)\inter(\theta_n,\tau_n+\f_n)-(\theta_0+\ddc\tau_0,\f_0)\inter(\theta_n+\ddc\tau_n,\f_n).
$$
By~\eqref{equ:enint}, we have 
$$
(0,\f_0)\cdot(\theta_1,\tau_1+\f_1)\inter(\theta_n,\tau_n+\f_n)-(0,\f_0)\cdot(\theta_1+\ddc\tau_1,\f_1)\inter(\theta_n+\ddc\tau_n,\f_n)=0. 
$$
This implies that $F(\f_0,\dots,\f_n)$ is independent of $\f_0$, and hence equal to $F(0,\f_1,\dots,\f_n)$. Applying the same argument successively to $\f_1,\dots,\f_n$, we end up with $F(\f_0,\dots,\f_0)=F(0,\dots,0)=(\theta_0,\tau_0)\inter(\theta_n,\tau_n)$, 
which proves~\eqref{equ:pairingtrans}. 
\end{proof}
By~\eqref{equ:enint}, the seminegativity property~\eqref{equ:semineg} translates into 
\begin{equation}\label{equ:negdef}
(0,\f)^2\cdot(\theta_1,\f_1)\inter(\theta_{n-1},\f_{n-1})\le 0
\end{equation}
for all $\f\in\cD$ and $\f_i\in\cD_{\theta_i}$. 

As in~\cite[\S 3.2]{trivval}, we further note the following straightforward monotonicity properties: 

\begin{prop}\label{prop:enpairing}
For all $\theta_i,\theta'_i\in\cZ$ and $\f_i,\f'_i\in\cD_{\theta_i}$, we have 
\begin{equation}\label{equ:enmon1}
\f_i\le\f'_i\text{ for all }i\Longrightarrow(\theta_0,\f_0)\inter(\theta_n,\f_n)\le(\theta_0,\f'_0)\inter(\theta_n,\f'_n);
\end{equation}
\begin{equation}\label{equ:enmon2}
\f_i\le 0\text{ and }0\le\theta_i\le\theta'_i\text{ for all }i\Longrightarrow (\theta'_0,\f_0)\inter(\theta'_n,\f_n)\le(\theta_0,\f_0)\inter(\theta_n,\f_n).
\end{equation}
\end{prop}
Combining~\eqref{equ:enmon1} and~\eqref{equ:encst}, we infer the Lipschitz property
\begin{equation}\label{equ:enlip}
\left|(\theta_0,\f_0)\inter(\theta_n,\f_n)-(\theta_0,\f'_0)\inter(\theta_n,\f'_n)\right|\le\sum_{i=0}^n\sup |\f_i-\f'_i|[\theta_0]\inter\widehat{[\theta_i]}\inter[\theta_n]
\end{equation} 
for all $\f_i,\f'_i\in\cD_{\theta_i}$. 

Following~\cite[\S 7.2]{trivval}, we next establish a lower bound for the energy pairing, which will play a crucial role in~\S\ref{sec:mixedMA} below.

\begin{thm}\label{thm:mixedenbound} Assume $\om_0,\dots,\om_n\in\cZ_+$ are commensurable, and set $\d:=\max_{i,j}\dT(\om_i,\om_j)$. If $0\ge \f_i\in\cD_{\om_i}$ for $i=1,\dots,n$, then 
$$
0\ge(\om_0,\f_0)\inter(\om_n,\f_n)\gtrsim e^{O(\d)}\min_i(\om_i,\f_i)^{n+1}. 
$$
\end{thm}
Here $\dT$ denotes the Thompson metric~\eqref{equ:dT}, and the implicit constant in $O(\d)$ only depends on $n$ (see~\S\ref{sec:notation}). 

\begin{lem}\label{lem:enconc} For any $\theta\in\cZ$, $\f\mapsto(\theta,\f)^{n+1}$ is concave on $\cD_\theta$. 
\end{lem}
\begin{proof} This is a formal consequence of~\eqref{equ:negdef}, see Appendix~\ref{sec:CS}. 
\end{proof}

\begin{lem}\label{lem:ennef} Pick $\om,\om'\in\cZ_+$ and $t\ge 1$ such that $\om\le\om'\le t\om$. For any  $\f\in\cD_\om\subset\cD_{\om'}$ such that $\f\le 0$, we then have
$$
0\ge(\om',\f)^{n+1}\ge t^n(\om,\f)^{n+1}.
$$ 
\end{lem}

\begin{proof} By~\eqref{equ:enmon2} we have 
$$
0\ge (\om',\f)^{n+1}\ge (t\om,\f)^{n+1}=t^{n+1}(\om,t^{-1}\f)^{n+1}. 
$$
Since $t^{-1}\in[0,1]$, concavity of the energy (Lemma~\ref{lem:enconc}) yields
$(\om,t^{-1}\f)^{n+1}\ge t^{-1}(\om,\f)^{n+1}$, and the result follows. 
\end{proof}

 \begin{lem}\label{lem:ensum}
Pick $\om_0,\dots,\om_r\in\cZ_+$ and $0\ge\f_i\in\cD_{\om_i}$ for $i=0,\dots,r$. Assume also given $t\ge 1$ such that $\om_i\le t\om_j$ for all $i,j$. Then 
\begin{equation}\label{equ:ensum}
(\sum_i\om_i,\sum_i\f_i)^{n+1}\ge C_{r,n} t^{rn}\sum_i(\om_i,\f_i)^{n+1}
\end{equation}
with $C_{r,n}:=\left(2^r r!\right)^n$. 
\end{lem}

\begin{proof} Assume first $r=1$. Set $\om:=\frac{t}{1+t}(\om_0+\om_1)$, and observe that $\om_0\le\om\le t\om_0$, $\om_1\le\om\le t\om_1$, and hence $\tfrac 12(\om_0+\om_1)\le\om\le\tfrac t2(\om_0+\om_1)$. Thus
\begin{multline*}
\left(\om_0+\om_1,\f_0+\f_1\right)^{n+1}=2^{n+1}\left(\tfrac 12(\om_0+\om_1),\tfrac 12(\f_0+\f_1)\right)^{n+1}\\
  \ge 2^{n+1}\left(\om,\tfrac 12(\f_0+\f_1)\right)^{n+1}
\ge 2^n\left((\om,\f_0)^{n+1}+(\om,\f_1)^{n+1}\right)\\
\ge (2t)^n\left((\om_0,\f_0)^{n+1}+(\om_1,\f_1)^{n+1}\right), 
\end{multline*}
where the first inequality holds by~\eqref{equ:enmon2}, the second one by Lemma~\ref{lem:enconc}, and the third one by Lemma~\ref{lem:ennef}. 

Assume now $r\ge 2$, and set $\om'_0:=\sum_{i>0}\om_i$, $\f'_0:=\sum_{i>0}\f_i$. Since $(rt)^{-1}\om_0\le\om'_0\le rt\om_0$, the first part of the proof yields 
$$
(\sum_i\om_i,\sum_i\f_i)^{n+1}=(\om_0+\om'_0,\f_0+\f'_0)^{n+1}\ge (2rt)^n\left((\om_0,\f_0)^{n+1}+(\om'_0,\f'_0)^{n+1}\right)
$$
By induction, we have on the other hand
$$
(\om'_0,\f'_0)^{n+1}\ge C_{r-1,n} t^{(r-1)n} \sum_{i>0}(\om_i,\f_i)^{n+1},
$$
The result follows, since $(\om_0,\f_0)^{n+1}\le 0$ and 
$$
C_{r,n} t^{rn}=(2rt)^n C_{r-1,n}\ge(2rt)^n. 
$$
\end{proof}

\begin{proof}[Proof of Theorem~\ref{thm:mixedenbound}] Expanding out $(\om_0+\dots+\om_n,\f_0+\dots+\f_n)^{n+1}$ yields
$$
(n+1)! (\om_0,\f_0)\inter(\om_n,\f_n)\ge (\om_0+\dots+\om_n,\f_0+\dots+\f_n)^{n+1}, 
$$
and we conclude by Lemma~\ref{lem:ensum} with $r=n$. 
\end{proof}

%
%
%
\subsection{The Monge--Amp\`ere operator and the submean value property}\label{sec:submean}
Pick $\om\in\cZ_+$ with $[\om]\in\Pos(X)$. We define its \emph{volume} as 
$$
V_\om:=\int\om^n=[\om]^n, 
$$
which is positive by Lemma~\ref{lem:posint}, and introduce the probability measure
$$
\mu_\om:=V_\om^{-1}\om^n. 
$$
\begin{defi} The \emph{Monge--Amp\`ere operator} $\MA_\om\colon\cD_\om\to\cM$ is defined by setting
$$
\MA_\om(\f):=V_\om^{-1}\om_\f^n. 
$$
\end{defi}
Equivalently, 
\begin{equation}\label{equ:MAmu}
\MA_\om(\f)=\mu_{\om_\f}.
\end{equation}

By~\eqref{equ:enlip}, the \emph{Chern--Levine--Nirenberg inequality} 
\begin{equation}\label{equ:CLN}
\left|\int\tau\left(\MA_\om(\f)-\MA_\om(\p)\right)\right|\le n\sup |\f-\p|
\end{equation}
is satisfied for all $\f,\p,\tau\in\cD_\om$.

We next consider the quantity
\begin{equation}\label{equ:submean1}
T_\om:=\sup\left\{\sup\f-\int\f\,\mu_\om\mid\f\in\cD_\om\right\}\in [0,+\infty]. 
\end{equation}
Thus $T_\om<\infty$ iff there exists $C>0$ such that the submean value inequality
\begin{equation}\label{equ:submean2}
\sup \f\le\int\f\,\mu_\om+C
\end{equation}
holds for all $\f\in\cD_\om$. 

\begin{prop}\label{prop:submean} If $T_\om<\infty$, then $T_{\om'}<\infty$ for any $\om'\in\cZ_+$ with $[\om']\in\Pos(X)$.
\end{prop}

\begin{defi}\label{defi:submean} We say that the \emph{submean value property} holds if $T_\om<\infty$ for some (hence any) $\om\in\cZ_+$ with $[\om]\in\Pos(X)$. 
\end{defi}
When $X$ is a compact K\"ahler or projective Berkovich space, this property holds iff $X$ is irreducible (see Theorem~\ref{thm:submean} below). 

\begin{lem}\label{lem:MV} Pick $\tau\in\cD_\om$, and $\om'\in\cZ_+$ commensurable to $\om$. Then
\begin{equation}\label{equ:Mddc}
T_{\om_\tau}\le T_\om+(n+2)\sup|\tau| 
\end{equation}
and
\begin{equation}\label{equ:Thom}
T_{\om'}\le e^{O(\d)} T_\om
\end{equation}
with $\d:=\dT(\om,\om')$. 
\end{lem}
Recall that, in this paper, the implicit constant in $O$ only depends on $n$ (see~\S\ref{sec:notation}). 

\begin{proof} Pick $\f\in\cD_{\om_\tau}$. Then $\f+\tau\in\cD_\om$, and the Chern--Levine--Nirenberg inequality~\eqref{equ:CLN} yields
$$
\f+\tau\le\int(\f+\tau)\MA_\om(0)+T_\om\le\int(\f+\tau)\MA_\om(\tau)+n\sup|\tau|+T_\om. 
$$
This proves~\eqref{equ:Mddc},  in view of~\eqref{equ:MAmu}. 

For each $t>0$ we have $\cD_{t\om}=t\cD_\om$, and hence $T_{t\om}=tT_\om$. Further, $\om'\le\om$ implies $\cD_{\om'}\subset\cD_\om$ and 
$$
V_{\om'}\mu_{\om'}=(\om')^n\le\om^n=V_\om\mu_\om,
$$ 
and hence $V_{\om'} T_{\om'}\le V_\om T_{\om}$. These properties imply~\eqref{equ:Thom}. 
\end{proof}

\begin{proof}[Proof of Proposition~\ref{prop:submean}] The condition $T_\om<\infty$ only depends on the commensurability class of $\om$, by~\eqref{equ:Thom}, and it only depends on $[\om]\in\Pos(X)$, by~\eqref{equ:Mddc}. This condition is thus independent of $\om$, since any two classes in $\Pos(X)$ admit commensurable representatives (see Proposition~\ref{prop:comm}). 
\end{proof}

\begin{thm}\label{thm:submean} Assume $X$ is either a compact K\"ahler space or a projective Berkovich space over a non-Archimedean field $k$. Then the submean value property holds iff $X$ is irreducible. 
\end{thm}

\begin{lem}\label{lem:submean} Let $\fa\subset\cO_X$ be a coherent ideal sheaf. Then we can find $\om\in\cZ_+$ with $[\om]\in\Pos(X)$ and a decreasing sequence $(\f_j)_j$ in $\cD_{\om}$ such that, for any given choice of local generators $(f_\a)$ of $\fa$ near a point of $X$, $\f_j-\max\{\log\max_\a|f_\a|,-j\}$ is locally bounded uniformly with respect to $j$. 
\end{lem}
\begin{proof} Assume first $X$ is compact K\"ahler. Pick a finite open cover $(U_i)$ of $X$ on which $\fa$ is generated by a finite subset $(f_{i\a})_\a$ of $\cO(U_i)$. Set 
$$
\f_{ij}:=\tfrac 12\log(\sum_\a|f_{i\a}|^2+e^{-2j}),
$$
and observe that, for all $i,k$, $|\f_{ij}-\f_{kj}|$ is uniformly bounded with respect to $j$ on each compact subset of $U_i\cap U_k$. Now choose a partition of unity subordinate to $(U_i)$, of the form $(\chi_i^2)$ with $\chi_i\in C^\infty_c(U_i)$, and set 
$$
\f_j:=\log\sum_i\chi_i^2 e^{\f_{ij}}. 
$$
As in~\cite[Lemma~3.5]{Demreg}, we find a K\"ahler form $\om$ such that $\ddc\f_j\ge -\om$ for all $j$, and the result follows. 

Assume now $X$ is a projective Berkovich space. Pick an ample line bundle $L$ such that $L\otimes\fa$ is generated by global sections $(s_\a)$, and choose a PL metric on $L$ with curvature form $\om\in\cZ_+$. Then setting $\f_j:=\max\{\log\max_\a|s_\a|,-j\}$ yields the result. 
\end{proof}

\begin{proof}[Proof of Theorem~\ref{thm:submean}] Assume first that the submean value property holds, and pick an $n$-dimensional irreducible component $Y$ of $X$. Applying Lemma~\ref{lem:submean} to the ideal sheaf of $Y$ in $X$ yields $\om\in\cZ_+$ with $[\om]\in\Pos(X)$ and a decreasing sequence $(\f_j)$ in $\cD_\om$ such that $\sup_Y\f_j\to-\infty$ while $(\f_j)$ is uniformly bounded on compact subsets of $X\setminus Y$. Since $\int_Y\om^n=[\om]^n\cdot [Y]>0$, we get $\int\f_j\,\mu_\om\to-\infty$, and hence $\sup \f_j\to-\infty$, by the submean value property. It follows that $X=Y$ is irreducible. 

Conversely, assume $X$ is irreducible. In the K\"ahler case, pick a K\"ahler form $\om$ on $X$, and choose a resolution of singularities $\pi\colon X'\to X$ with $X'$ a connected compact K\"ahler manifold, and pick a K\"ahler form $\om'$ on $X'$ such that $\om'\ge\pi^\star\om$. Since $\mu:=V_\om^{-1}\pi^\star\om^n=f\om'^n$ satisfies $\PSH(X',\om')\subset L^1(\mu)$, \cite[Proposition~2.7]{GZ} yields $C>0$ such that $\sup_{X'}\p\le\int\p\,\mu+C$ for all $\p\in\PSH(X',\om')$. Applying this to $\p=\pi^\star\f$ with $\f\in\cD_\om$ yields the submean value property. 

In the non-Archimedean case, pick an ample line bundle $L$ and a PL metric on $L$ determined by an ample model/test configuration $(\cX,\cL)$ of $(X,L)$, with curvature form $\om\in\cZ_+$. Each $\f\in\cD_\om$ satisfies $\sup \f=\max_\Ga\f$, where $\Ga\subset X$ is the (finite) set of Shilov points attached to $\cX$ (see~\cite[Proposition~4.22]{GM} or~\cite[Lemma~6.3]{BE}). On the other hand, $\Ga$ is also the support of $\mu_\om$, by definition of the measure $\om^n$ in terms of intersection numbers. Now~\cite[Theorem~2.21]{trivvalold} yields $C>0$ such that $|\f(x)-\f(y)|\le C$ for all $x,y\in\Ga$ and all $\f\in\cD_\om$, and we infer, as desired, $\sup \f\le\int\f\,\mu_\om+C$ for all $\f\in\cD_\om$.  
\end{proof}

%
%
%
\subsection{Monge--Amp\`ere energy and the Dirichlet functional}\label{sec:enfunc}
From now on we fix $\om\in\cZ_+$ with $[\om]\in\Pos(X)$. 

\begin{defi} The \emph{Monge--Amp\`ere energy} $\en_\om\colon\cD\to\R$ is defined by 
\begin{equation}\label{equ:Eom}
\en_\om(\f):=\frac{(\om,\f)^{n+1}}{(n+1)V_\om}. 
\end{equation}
\end{defi}
For all $\f,\p\in\cD$ we have 
\begin{equation}\label{equ:Evar}
\en_\om(\f)-\en_\om(\p)=\frac{1}{n+1}\sum_{j=0}^n V_\om^{-1}\int(\f-\p)\,\om_\f^j\wedge\om_\p^{n-j}. 
\end{equation}
By~\eqref{equ:enint}, this amounts indeed to the basic identity 
$$
(\om,\f)^{n+1}-(\om,\p)^{n+1}=(0,\f-\p)\cdot\sum_{j=0}^n(\om,\f)^j\cdot(\om,\p)^{n-j}. 
$$
Assume now $\f,\p\in\cD_\om$. By~\eqref{equ:Evar}, we then have 
\begin{equation}\label{equ:Emon}
\f\le\p\Longrightarrow\en_\om(\f)\le\en_\om(\p).
\end{equation}
Further, $\en_\om((1-t)\f+t\p)$ is a polynomial function of $t\in [0,1]$, with 
\begin{equation}\label{equ:deren}
\frac{d}{dt}\bigg|_{t=0}\en_\om((1-t)\f+t\p)=\int(\p-\f)\MA_\om(\f). 
\end{equation}
This characterizes $\en_\om$ as the unique primitive of the Monge--Amp\`ere operator that vanishes at $0\in\cD_\om$. By Lemma~\ref{lem:enconc}, $\en_\om$ is concave on $\cD_\om$, which translates into
\begin{equation}\label{equ:enconc}
\en_\om(\p)\le\en_\om(\f)+\int(\p-\f)\MA_\om(\f)
\end{equation}
for all $\f,\p\in\cD_\om$. 

\begin{lem}\label{lem:entrans} For any $\f\in\cD$, we have 
\begin{equation}\label{equ:entrans}
\en_{\om_\tau}(\f)=\en_{\om}(\f+\tau)-\en_\om(\tau).
\end{equation}
Further, $\f\in\cD_{\om_\tau}\Longleftrightarrow\f+\tau\in\cD_\om$. 
\end{lem}
\begin{proof} The first point follows from~\eqref{equ:pairingtrans}. Further, $\f\in\cD_{\om_\tau}$ iff $\om+\ddc\tau+\ddc\f\ge 0$, which is also equivalent to $\f+\tau\in\cD_\om$. 
\end{proof}

\begin{defi} We define the \emph{Dirichlet functional} $\jj_\om\colon\cD_\om\times\cD_\om\to\R_{\ge 0}$ by setting 
$$
\jj_\om(\f,\p)=\jj_\om(\f,\p):=\en_\om(\f)-\en_\om(\p)+\int(\p-\f)\MA_\om(\f). 
$$
\end{defi}
By concavity of $\en_\om$, $\jj_\om(\f,\p)$ is a convex function of $\p\in\cD_\om$. Note also that 
\begin{equation}\label{equ:IJ}
\jj_\om(\f,\p)+\jj_\om(\p,\f)=\int(\f-\p)(\MA_\om(\p)-\MA_\om(\f)).
\end{equation}
By~\eqref{equ:dapp}, we more explicitly have
\begin{equation}\label{equ:Dirichlet}
\jj_\om(\f,\p)=V_\om^{-1}\sum_{j=0}^{n-1}\frac{j+1}{n+1}\int(\f-\p)\,\ddc(\p-\f)\wedge\om_\f^j\wedge\om_\p^{n-1-j}. 
\end{equation}
We simply write
\begin{equation}\label{equ:J}
\jj_\om(\f):=\jj_\om(0,\f)=\int\f\,\mu_\om-\en_\om(\f). 
\end{equation}

\begin{exam} When $n=1$, we have
$$
\jj_\om(\f,\p)=\jj_\om(\p,\f)=\tfrac12\int(\f-\p)\ddc(\p-\f),
$$
which recovers the usual expression for the Dirichlet functional on a Riemann surface. 
\end{exam}

For each $R>0$ we set
\begin{equation}\label{equ:DR}
\cD_{\om,R}:=\{\f\in\cD_\om\mid\jj_\om(\f)\le R\}. 
\end{equation} 
\begin{lem}\label{lem:Jspan} For each $R>0$, $\cD_{\om,R}$ is a convex subset of $\cD$ that generates it. 
\end{lem} 
\begin{proof} By convexity of $\jj_\om=\jj_\om(0,\cdot)$, $\cD_{\om,R}$ is convex. Since $\cD_\om$ spans $\cD$ (see Corollary~\ref{cor:pshspan}), to see that $\cD_{\om,R}$ spans it suffices to show that any $\f\in\cD_\om$ satisfies $\jj_\om(t\f)\le R$ for $0<t\ll 1$, which holds since $\jj_\om(t\f)$ is a polynomial function of $t$ with $\jj_\om(0)=0$.  
\end{proof}

We may now collect the fundamental properties of the Dirichlet functional in the next result. 

\begin{thm}\label{thm:est} For all $\f,\f',\p,\p',\tau\in\cD_\om$ and $t\in[0,1]$, the following holds:
\begin{itemize}
\item \emph{quasi-symmetry:} 
\begin{equation}\label{equ:IJsum}
\jj_\om(\f,\p)\approx\jj_\om(\p,\f); 
\end{equation}
\item \emph{quasi-triangle inequality:}
\begin{equation}\label{equ:qtri}
\jj_\om(\f,\p)\lesssim\jj_\om(\f,\tau)+\jj_\om(\tau,\p);
\end{equation}
\item \emph{quadratic estimate:}
\begin{equation}\label{equ:Jquad}
\jj_\om(\f,(1-t)\f+t\p)\lesssim t^2 \jj_\om(\f,\p);
\end{equation}
\item \emph{uniform concavity:} 
$$
\en_\om((1-t)\f+t\p)-[(1-t)\en_\om(\f)+t \en_\om(\p)]\gtrsim t (1-t)\jj_\om(\f,\p). 
$$
\end{itemize}
For all $\f,\f',\p,\p'\in\cD_{\om,R}$, we further have the following H\"older estimates: 
\begin{equation}\label{equ:holdMA}
\left|\int(\f-\f')\left(\MA_\om(\p)-\MA_\om(\p')\right)\right|\lesssim\jj_\om(\f,\f')^\a \jj_\om(\p,\p')^{1/2} R^{1/2-\a}; 
\end{equation}
and
\begin{equation}\label{equ:holdJ}
\left|\jj_\om(\f,\p)-\jj_\om(\f',\p')\right|\lesssim\max\{\jj_\om(\f,\f'),\jj_\om(\p,\p')\}^\a R^{1-\a}, 
\end{equation}
where $\a:=2^{-n}$. 
\end{thm}

\begin{proof} In view of~\eqref{equ:negdef}, this is a direct consequence of Theorem~\ref{thm:CS} applied to
\begin{itemize}
\item the vector space $V:=\cZ\times\cD$ with projection $\pi\colon V\to\cZ$ onto the first factor;
\item the convex cone $P:=\{(\theta,\f)\in V\mid\f\in\cD_\theta\}$; 
\item the homogeneous polynomial $F\colon V\to\R$ defined by $F(\theta,\f):=-\frac{(\theta,\f)^{n+1}}{(n+1)V_\om}$. 
\end{itemize}
\end{proof}

As a simple consequence of~\eqref{equ:entrans}, we finally note: 
\begin{lem}\label{lem:Jtrans} For each $\tau\in\cD_\om$ and $\f,\p\in\cD_{\om_\tau}$ we have $\f+\tau,\p+\tau\in\cD_\om$ and 
\begin{equation}\label{equ:Jtrans}
\jj_{\om_\tau}(\f,\p)=\jj_\om(\f+\tau,\p+\tau). 
\end{equation}
\end{lem} 

\begin{rmk} The above formalism recovers that of~\cite{BBGZ,BBEGZ,nama,trivval} in the  K\"ahler and non-Archimedean settings. However, in contrast to those works, we do not explicitly introduce here the functional $\ii_\om(\f,\p)$, which corresponds to the right-hand side of~\eqref{equ:IJ}.
\end{rmk}
%
%
%
\section{Measures of finite energy}\label{sec:mes} 
In what follows, we pick $\om\in\cZ_+$ with $[\om]\in\Pos(X)$. We define the space $\cM^1_\om$ of measures of finite energy with respect to $\om$, and show, assuming a certain orthogonality property, that it is complete with respect to a quasi-metric $\d_\om$ induced by the Dirichlet functional. 
%
%
%
\subsection{The energy of a measure}

\begin{defi}\label{defi:enmes} We define the \emph{energy of $\mu\in\cM$ relative to $\p\in\cD_\om$} as 
\begin{equation}\label{equ:Jmuvar}
\jj_\om(\mu,\p):=\sup_{\f\in\cD_\om}\{\en_\om(\f)-\en_\om(\p)+\int(\p-\f)\mu\}\in[0,+\infty]. 
\end{equation}
\end{defi}
The choice of notation is justified by~\eqref{equ:JMA} below. When $\p=0$, we simply write 
\begin{equation}\label{equ:envar}
\jj_\om(\mu):=\jj_\om(\mu,0)=\sup_{\f\in\cD_\om}\{\en_\om(\f)-\int\f\,\mu\}, 
\end{equation}
and call it the \emph{energy\footnote{This corresponds to $\en_\om^\vee(\mu)$ in the notation of~\cite{BBGZ,trivval}, and to $\|\mu\|_\om$ in that of~\cite{nakstab2}.} of $\mu$} (with respect to $\om$). Note that 
\begin{equation}\label{equ:envar2}
\jj_\om(\mu)=\sup_{\f\in\cD_\om}\{\int\f(\mu_\om-\mu)-\jj_\om(\f)\}, 
\end{equation}
by~\eqref{equ:J}. 
\begin{prop}\label{prop:Jmu} For each $\p\in\cD_\om$, the functional $\jj_\om(\cdot,\p)\colon\cM\to[0,+\infty]$ is convex and weakly lsc, and satisfies, for all $\mu\in\cM$ and $\f\in\cD_\om$, 
\begin{equation}\label{equ:Jmu}
\jj_\om(\mu,\p)=\jj_\om(\mu)+\int\p\,\mu-\en_\om(\p);
\end{equation}
\begin{equation}\label{equ:JMA}
\jj_\om(\MA_\om(\f),\p)=\jj_\om(\f,\p);
\end{equation} 
\begin{equation}\label{equ:JMA2}
\jj_\om(\MA_\om(\f))\approx\jj_\om(\f);
\end{equation}
\begin{equation}\label{equ:Jmutri}
\jj_\om(\f,\p)\lesssim\jj_\om(\mu,\f)+\jj_\om(\mu,\p). 
\end{equation} 
\end{prop}
\begin{proof} Convexity and lower semicontinuity are clear from~\eqref{equ:Jmuvar}, which also directly yields~\eqref{equ:Jmu}. By~\eqref{equ:enconc}, any $\f\in\cD_\om$ further computes the supremum defining $\jj_\om(\MA_\om(\f),\p)$, which is thus equal to 
$$
\en_\om(\f)-\en_\om(\p)+\int(\p-\f)\MA_\om(\f)=\jj_\om(\f,\p).
$$ 
This proves~\eqref{equ:JMA}, which implies 
$$
\jj_\om(\MA_\om(\f))=\jj_\om(\MA_\om(\f),0)=\jj_\om(\f,0)\approx\jj_\om(0,\f)=\jj_\om(\f),
$$
see~\eqref{equ:IJsum}. Finally, pick $\mu\in\cM$, and set $\tau=\tfrac 12(\f+\p)\in\cD_\om$. By~\eqref{equ:Jmuvar}, we have 
$$
\jj_\om(\mu,\f)\ge\en_\om(\tau)-\en_\om(\f)+\int(\f-\tau)\,\mu,\quad\jj_\om(\mu,\p)\ge\en_\om(\tau)-\en_\om(\p)+\int(\p-\tau)\,\mu, 
$$
and hence
$$
\jj_\om(\mu,\f)+\jj_\om(\mu,\p)\ge2\en_\om(\tau)-(\en_\om(\f)+\en_\om(\p)). 
$$
On the other hand, 
$$
2\en_\om(\tau)-(\en_\om(\f)+\en_\om(\p))\gtrsim\jj_\om(\f,\p)
$$
by uniform concavity of $\en_\om$ (see Theorem~\ref{thm:est}), and~\eqref{equ:Jmutri} follows. 
\end{proof}

Generalizing Lemma~\ref{lem:Jtrans}, we note: 
\begin{lem}\label{lem:Jmestrans} For all $\tau\in\cD_\om$, $\p\in\cD_{\om_\tau}$ and $\mu\in\cM$ we have 
\begin{equation}\label{equ:Jmestrans}
\jj_{\om_\tau}(\mu,\p)=\jj_\om(\mu,\p+\tau). 
\end{equation}
In particular, 
\begin{equation}\label{equ:Jmestrans2}
\jj_{\om_\tau}(\mu)=\jj_\om(\mu)+\int\tau\,\mu-\en_\om(\tau).
\end{equation}
\end{lem}
\begin{proof} By Lemma~\ref{lem:entrans}, we have $\f\in\cD_{\om_\tau}\Leftrightarrow\f+\tau\in\cD_\om$, and
$$
\en_{\om_\tau}(\f)-\en_{\om_\tau}(\p)+\int(\f-\p)\mu=\en_\om(\f+\tau)-\en_\om(\p+\tau)+\int((\f+\tau)-(\p+\tau))\mu. 
$$
Taking the sup over $\f$ yields~\eqref{equ:Jmestrans}, and~\eqref{equ:Jmestrans2} follows, by~\eqref{equ:Jmu}. 
\end{proof}

\begin{rmk}\label{rmk:Jnonpos} If we drop the assumption that $\om\ge 0$, but still require $[\om]\in\Pos(X)$, then $\en_\om(\f)$ and $\jj_\om(\mu)$ can still defined by~\eqref{equ:Eom} and~\eqref{equ:envar}, respectively. Then Lemma~\ref{lem:entrans}, and hence~\eqref{equ:Jmestrans2}, remain valid for any $\tau\in\cD$. This will only get used in the context of Theorem~\ref{thm:diffen} below. 
\end{rmk}

%
%
%
\subsection{Measures of finite energy}

\begin{defi}\label{defi:M1} The space of \emph{measures of finite energy (with respect to $\om$)} is defined as
$$
\cM^1_\om:=\{\mu\in\cM\mid\jj_\om(\mu)<\infty\}.
$$
It is endowed with the \emph{strong topology}, defined as the coarsest refinement of the weak topology in which $\jj_\om\colon\cM^1_\om\to\R_{\ge 0}$ becomes continuous. 
\end{defi}
In other words, a net $(\mu_i)$ converges strongly to $\mu$ in $\cM^1_\om$ iff $\mu_i\to\mu$ weakly in $\cM$ and $\jj_\om(\mu_i)\to\jj_\om(\mu)$. For any $R>0$ we also set
\begin{equation}\label{equ:M1J}
 \cM^1_{\om,R}:=\{\mu\in\cM^1_\om\mid\jj_\om(\mu)\le R\}.
 \end{equation}
By Proposition~\ref{prop:Jmu}, this set is convex and weakly compact. By~\eqref{equ:Jmu}, 
$$
\jj_\om(\cdot,\p)\colon\cM^1_\om\to\R_{\ge 0}
$$ 
is continuous in the strong topology for any $\p\in\cD_\om$. By Lemma~\ref{lem:Jmestrans}, this yields: 

\begin{prop}\label{prop:M1ind} The topological space $\cM^1_\om$ only depends on the positive class $[\om]\in\Pos(X)$.
\end{prop}
One should be wary of the fact that, in the present generality, even a `nice' probability measure of the form $\mu=\theta_1\winter\theta_n$ with $\theta_i\in\cZ_+$ need not be of finite energy with respect to $\om$ in general (see however Theorem~\ref{thm:mixedMA} below): 

\begin{exam}\label{exam:masscompo} Let $X$ be either a compact K\"ahler or projective Berkovich space. For each irreducible component $Y$ of $X$ and each $\mu\in\cM^1_\om$, we then have 
\begin{equation}\label{equ:masscomp}
\mu(Y)=\frac{[\om|_Y]^n}{[\om]^n}. 
\end{equation}
Indeed, this is proved in~\cite[Corollary~9.13]{trivval} in the trivially valued case, and the proof can be adapted to the general case. Now~\eqref{equ:masscomp} fails in general for $\mu=\mu_{\om'}$ with $\om'\in\cZ_+$ such that $[\om]\ne [\om']\in\Pos(X)$, and hence $\cM^1_{\om'}\ne\cM^1_\om$. 
\end{exam} 

By~\eqref{equ:envar2}, we have, for all $\f\in\cD_\om$ and $\mu\in\cM$, 
$$
\int\f\,(\mu_\om-\mu)\le\jj_\om(\f)+\jj_\om(\mu).
$$
The following converse will come in handy. 

\begin{lem}\label{lem:finen} Assume that $\mu\in\cM$ satisfies
$$
S:=\sup_{\f\in\cD_{\om,R}}\int\f(\mu_\om-\mu)<\infty
$$
for some $R>0$. Then $\mu$ has finite energy, and 
\begin{equation}\label{equ:JS}
\jj_\om(\mu)\lesssim S(1+R^{-1}S).
\end{equation}
\end{lem}
\begin{proof} Pick $\f\in\cD_\om$ and set $J:=\jj_\om(\f)$. By~\eqref{equ:Jquad}, we have $\jj_\om(t\f)\lesssim t^2 J$ for any $t\in [0,1]$, and we can thus choose 
$$
1\le a\lesssim 1+(R^{-1} J)^{1/2}
$$
such that $\jj_\om(a^{-1}\f)\le R$. By assumption, we then have $\int a^{-1}\f\,(\mu_\om-\mu)\le S$, and hence 
$$
\en_\om(\f)-\int\f\,\mu=\int\f\,(\mu_\om-\mu)-\jj_\om(\f)\le a S-J
$$
$$
\lesssim S+S R^{-1/2}J^{1/2}-J\le S+\tfrac 14 S^2 R^{-1}, 
$$
where the last inequality follows from the elementary estimate $\sup_{y\ge 0}(x y^{1/2}-y)=x^2/4$ for any $x\ge 0$. Taking the supremum over $\f$ yields~\eqref{equ:JS}. 
\end{proof}

%
%
\subsection{Legendre transform of the energy} 

Here we compute the Legendre transform of the convex functional $\jj_\om=\jj_\om(\cdot,0)\colon\cM\to[0,+\infty]$. 

\begin{defi}\label{defi:ten} For any $f\in\Cz(X)$ we set 
$$
\ten_\om(f):=\sup_{f\ge\f\in\cD_\om}\en_\om(\f).
$$
\end{defi}
By monotonicity of $\en_\om$ on $\cD_\om$ (see~\eqref{equ:Emon}), the functional $\ten_\om\colon\Cz(X)\to\R$ so defined restricts to $\en_\om$ on $\cD_\om$. Like the latter, $\ten_\om$ is further concave, monotone increasing, and equivariant with respect to translation, \ie 
$$
\ten_\om(f+c)=\ten_\om(f)+c\text{ for }c\in\R.
$$

\begin{prop}\label{prop:enleg} For all $f\in\Cz(X)$ and $\mu\in\cM$ we have 
\begin{equation}\label{equ:legdual}
\ten_\om(f)=\inf_{\nu\in\cM}\{\jj_\om(\nu)+\int f\,\nu\};\quad\jj_\om(\mu)=\sup_{g\in\Cz(X)}\{\ten_\om(g)-\int g\,\mu\}.
\end{equation}

\end{prop}
\begin{proof} Define the (convex) Legendre transform $\ten_\om^\vee\colon\Cz(X)^\vee\to\R\cup\{+\infty\}$ as the right-hand side of~\eqref{equ:legdual}, \ie 
$$
\ten_\om^\vee(\mu):=\sup_{g\in\Cz(X)}\{\ten_\om(g)-\int f\,\mu\}. 
$$
Since $\ten_\om$ is increasing and equivariant, it is straightforward to see that $\ten_\om^\vee(\mu)<\infty$ implies $\mu\ge 0$ and $\int\mu=1$, \ie $\mu\in\cM$ (compare~\cite[Proposition~9.8]{trivval}). By Legendre duality, the result is thus equivalent to $\ten_\om^\vee(\mu)=\jj_\om(\mu)$ for $\mu\in\cM$. Since $\ten_\om$ restricts to $\en_\om$ on $\cD_\om$, we trivially have $\ten_\om^\vee(\mu)\ge\jj_\om(\mu)$. Conversely, pick $f\in\Cz(X)$ and $\f\in\cD_\om$ with $\f\le f$. Then 
$$
\jj_\om(\mu)\ge\en_\om(\f)-\int\f\,\mu\ge\en_\om(\f)-\int f\,\mu,
$$
where the first and second inequality respectively follow from \eqref{equ:envar} and $\f\le f$. Taking the supremum over $\f$ and then over $f$ yields $\jj_\om(\mu)\ge\ten_\om(f)-\int f\,\mu$ and $\jj_\om(\mu)\ge\ten_\om^\vee(\mu)$. 
\end{proof}

%
%
%
\subsection{Orthogonality and differentiability}

\begin{defi} We say that $\cD_\om$ \emph{admits maxima} if, for all $\f,\p\in\cD_\om$ and $f\in\cD$ such that $\max\{\f,\p\}<f$ pointwise on $X$, there exists $\tau\in\cD_\om$ with $\max\{\f,\p\}\le\tau<f$. 
\end{defi}
This equivalently means that, for any $f\in\cD$, the poset 
$$
\cD_{\om,<f}:=\{\f\in\cD_\om\mid\f<f\}
$$ 
is inductive. We can then consider limits of nets indexed by $\cD_{\om,<f}$. For instance, note that
\begin{equation}\label{equ:tenlim}
\ten_\om(f)=\lim_{\f\in\cD_{\om,<f}}\en_\om(\f). 
\end{equation}

\begin{exam} If $X$ is a compact K\"ahler space, then $\cD_\om$ admits maxima: take $\tau:=\widetilde\max(\f,\p)$ for an appropriate regularized max function $\widetilde\max$. 
\end{exam}

\begin{exam} If $X$ is a projective Berkovich space, then $\cD_\om$ also admits maxima, since $\Q$-PL functions in $\cD_\om$ are dense in $\cD_\om$, and stable under max.
\end{exam}

\begin{rmk} While we will not pursue this direction here, one can introduce as in~\cite{siminag,BE,trivval} the space $\PSH(\om)$ of $\om$-psh functions $\f\colon X\to\R\cup\{-\infty\}$, defined as usc functions that can be obtained as pointwise limits of decreasing nets in $\cD_\om$, and such that $\int\f\,\mu_\om>-\infty$. Then $\cD_\om$ admits maxima iff $\PSH(\om)$ (or, equivalently, the subspace $\CPSH(\om):=\PSH(\om)\cap\Cz(X)$ of continuous $\om$-phs functions) is stable under max. 
\end{rmk}

\begin{defi}\label{defi:ortho} We say that $\om$ has the \emph{orthogonality property} if $\cD_\om$ admits maxima and 
\begin{equation}\label{equ:ortho}
\lim_{\f\in\cD_{\om,<f}}\int(f-\f)\MA_\om(\f)=0
\end{equation}
for all $f\in\cD$. 
\end{defi}
Explicitly, this means for that for any $\e>0$ there exists $\f_0\in\cD_{\om}$ such that $\f_0<f$ and $\int(f-\f)\MA_\om(\f)\le\e$ for all $\f\in\cD_{\om}$ with $\f_0\le\f<f$. 

\begin{rmk} The orthogonality property for $\om$ only depends on $[\om]\in\Pos(X)$. 
Indeed, for any $\tau\in\cD_\om$ and $f\in\cD$ we have  
\begin{equation}\label{equ:Domtau}
\f\in\cD_{\om_\tau,<f}\Longleftrightarrow\f+\tau\in\cD_{\om,<f+\tau}
\end{equation}
This implies that $\cD_\om$ has finite maxima iff $\cD_{\om_\tau}$ does, and similarly for the orthogonality property, using 
$$
\MA_{\om_\tau}(\f)=\MA_\om(\f+\tau)\quad\text{and}\quad\int(f-\f)\MA_{\om_\tau}(\f)=\int\left((f+\tau)-(\f+\tau)\right)\MA_{\om}(\f+\tau).
$$
\end{rmk}

\begin{exam}\label{exam:AOK} Assume $X$ is a compact K\"ahler space. Conjecturally, the orthogonality property always holds. This is known when $X$ is nonsingular, or $X$ is projective and $[\om]\in\Amp(X)$, see Appendix~\ref{sec:orthocomplex} for a more detailed discussion. 
\end{exam} 

\begin{rmk} More generally, the orthogonality property also holds in the setting of Remark~\ref{rmk:npos} (assuming $X$ is a compact K\"ahler manifold). This follows indeed from~\cite[Theorem~3.2]{LN}. 
\end{rmk}

Recall from Corollary~\ref{cor:pshspan} that any test function $f\in\cD$ can be written as 
\begin{equation}\label{equ:fdiff}
f=f^+-f^-,\quad f^\pm\in\cD_{C\om}
\end{equation}
for some $C=C(f)>0$. In line with~\cite[\S 7]{nama}, we show: 

\begin{prop}\label{prop:ortho} Assume $\cD_\om$ admits maxima. The following properties are then equivalent:
\begin{itemize}
\item[(i)] $\om$ has the orthogonality property; 
\item[(ii)] for any $f\in\cD$ written as~\eqref{equ:fdiff} for a given $C>0$, we have 
\begin{equation}\label{equ:unifdiff1}
\left|\ten_\om(\f+f)-\en_\om(\f)-\int f\,\MA_\om(\f)\right|\lesssim C\sup|f|
\end{equation}
for all $\f\in\cD_\om$; 
\item[(iii)] in the setting of (ii), we have 
\begin{equation}\label{equ:unifdiff2}
\left|\ten_\om(\f+t f)-\en_\om(\f)-t\int f\,\MA_\om(\f)\right|\lesssim Ct^2 \sup|f|
\end{equation}
for all $\f\in\cD_\om$ and $t\in\R$. 
\end{itemize}
\end{prop}

\begin{proof}  Assume (i). Write $f\in\cD$ as in~\eqref{equ:fdiff}, and pick $\f\in\cD_\om$. For any $\p\in\cD_{\om,<\f+\e f}$, \eqref{equ:enconc} yields 
$$
\int(\p-\f)\MA_\om(\p)\le\en_\om(\p)-\en_\om(\f)\le\int(\p-\f)\MA_\om(\f)\le\int f\,\MA_\om(\f). 
$$
By~\eqref{equ:CLN}, we also have $\left|\int f (\MA_\om(\f)-\MA_\om(\p))\right|\le 2n C\sup|\f-\p|$, and we infer
\begin{equation}\label{equ:qunidiff}
\left|\en_\om(\p)-\en_\om(\f)-\int f\,\MA_\om(\f)\right|\le\int\left((\f+f)-\p\right)\MA_\om(\p)+2n C\sup|\f-\p|.
\end{equation}
Now $\lim_{\p\in\cD_{\om,<\f+ f}}\en_\om(\p)=\ten_\om(\f+f)$ (see~\eqref{equ:tenlim}), while orthogonality yields
$$
\lim_{\p\in\cD_{\om,<\f+f}}\int\left((\f+f)-\p\right)\MA_\om(\p)=0.
$$
Further, any $\p\in\cD_{\om,<\f+f}$ large enough is greater than $\f-\sup|f|\in\cD_{\om,<\f+f}$, and hence satisfies $\sup|\f-\p|\le\sup|f|$. As a result, \eqref{equ:qunidiff} implies
$$
\left|\ten_\om(\f+ f)-\en_\om(\f)-\int f\,\MA_\om(\f)\right|\le 2n C\sup|f|, 
$$
which shows (i)$\Rightarrow$ (ii). Next, (ii)$\Rightarrow$(iii), since $f=f^+-f^-$ with $f\in\cD_{C\om}$ implies $tf=\mathrm{sgn}(t)(|t| f^+-|t|f^-)$ with $|t|f^\pm\in\cD_{C|t|\om}$.  

Finally, assume (iii), and pick $f\in\cD$. We need to show that 
$$
L:=\limsup_{\f\in\cD_{\om,<f}}\int(f-\f)\MA_\om(\f)\ge 0
$$
vanishes. Write $f$ as in~\eqref{equ:fdiff} for some $C>0$. Pick also $\f\in\cD_{\om,<f}$, and set $g:=f-\f\in\cD$. For any $t\in[0,1]$ we have $\f+t g=(1-t)\f+t f\le f$, and hence $\ten_\om(\f+t g)\le\ten_\om(f)$. On the other hand, since $f-\f=f^+-(f^-+\f)$ with $f^+,<f^-+\f\in\cD_{(C+1)\om}$, (iii) yields a constant $A>0$ only depending on $\f$ such that 
$$
0\le t\int(f-\f)\MA_\om(\f)\le\ten_\om(\f+t g)-\en_\om(\f)+t^2 A\sup|g|\le \ten_\om(f)-\en_\om(\f)+t^2 A\sup|g|. 
$$
Since any $\f\in\cD_{\om,<f}$ large enough is greater then $f^+-\sup f^-$, it satisfies 
$$
0\le g=f-\f\le\sup f^--f^-\le B
$$
with $B$ only depending on $f$. We infer $0\le tL\le t^2 AB$. Dividing by $t>0$ and letting $t\to 0_+$ yields, as desired, $L=0$. This proves (iii)$\Rightarrow$(i). 
\end{proof}

\begin{exam}\label{exam:AONA} Assume $X$ is a projective Berkovich space over a non-Archimedean field. Then \cite[Theorem~A]{BE} combined with the uniform differentiability estimate of~\cite[Lemma~3.2]{BGM} shows that~\eqref{equ:unifdiff1} is satisfied (compare~\cite[Lemma~8.7]{trivval}). By Proposition~\ref{prop:ortho}, it follows that the orthogonality property always holds in this setting.  
\end{exam}

%
%
%
\subsection{Maximizing sequences}

\begin{defi} We say that a sequence $(\p_i)$ in $\cD_\om$ is \emph{maximizing for $\mu\in\cM^1_\om$} if it computes the energy of $\mu$~\eqref{equ:envar}, \ie $\en_\om(\p_i)-\int\p_i\,\mu\to\jj_\om(\mu)$. 
\end{defi}
Equivalently, $(\p_i)$ is maximizing for $\mu$ iff $\jj_\om(\mu,\p_i)\to 0$, see~\eqref{equ:Jmu}. 

\begin{exam}\label{exam:maxMA} For any $\f\in\cD_\om$ the constant sequence $\p_i=\f$ is maximizing for $\mu=\MA_\om(\f)$ (see~\eqref{equ:JMA}). 
\end{exam}

As a key consequence of Proposition~\ref{prop:ortho}, we show: 

\begin{thm}\label{thm:maxcv} Assume $\om$ has the orthogonality property. Pick $\mu\in\cM^1_\om$ and a maximizing sequence $\p_i\in\cD_\om$. Then the measures $\mu_i:=\MA_\om(\p_i)$ converge strongly to $\mu$ in $\cM^1_\om$, \ie $\mu_i\to\mu$ weakly and $\jj_\om(\mu_i)\to\jj_\om(\mu)$. In particular, the image of the Monge--Amp\`ere operator 
$$
\MA_\om\colon\cD_\om\to\cM^1_\om
$$
is dense in the strong topology. 
\end{thm}

\begin{proof} Pick $f\in\cD$, and choose $C>0$ such that $f=f^+-f^-$ with $f^\pm\in\cD_{C\om}$. Since we assume orthogonality, Proposition~\ref{prop:ortho} yields $A>0$ such that 
$$
\left|\ten_\om(\p_i+t f)-\en_\om(\p_i)-t\int f\,\mu_i\right|\le At^2\sup|f|
$$
for all $i$ and $t>0$. By Proposition~\ref{prop:enleg} and~\eqref{equ:Jmu}, we have, on the other hand,
$$
\ten_\om(\p_i+t f)\le\jj_\om(\mu)+\int(\p_i+t f)\mu=\jj_\om(\mu,\p_i)+\en_\om(\p_i)+t\int f\,\mu. 
$$
Combining these estimates, we get
$$
t\int f\,\mu_i\le t\int f\,\mu+\jj_\om(\mu,\p_i)+A t^2.
$$
Since $\jj_\om(\mu,\p_i)\to 0$, we infer 
$$
t\limsup_i\int f\,\mu_i\le t\int f\,\mu+A t^2.
$$
Dividing by $t$ and letting $t\to 0_+$ yields $\limsup_i\int f\,\mu_i\le\int f\,\mu$. Replacing $f$ with $-f$, we get $\lim_i\int f\,\mu_i=\int f\,\mu$. By density of $\cD$ in $\Cz(X)$, this shows $\mu_i\to\mu$ weakly. 

For each $i$ we have $\jj_\om(\mu_i)=\en_\om(\p_i)-\int\p_i\,\mu_i$ (see~\eqref{equ:JMA}), and $\en_\om(\p_i)-\int\p_i\,\mu\to\jj_\om(\mu)$, since $(\p_i)$ is maximizing for $\mu$. It only remains to prove $\int\p_i\,(\mu_i-\mu)\to 0$. Since $\jj_\om(\p_i)$ is bounded (see~\eqref{equ:Jmutri},~\eqref{equ:holdMA} yields $C>0$ such that  
$$
\left|\int\p_i\,(\mu_i-\mu_j)\right|\lesssim C\jj_\om(\p_i,\p_j)^\a
$$
for all $i,j$, and hence 
$$
\left|\int\p_i\,(\mu_i-\mu_j)\right|\lesssim C\max\{\jj_\om(\mu,\p_i),\jj_\om(\mu,\p_j)\}^\a,
$$
by~\eqref{equ:Jmutri}. Since $\mu_j\to\mu$ weakly and $\jj_\om(\mu,\p_j)\to 0$ as $j\to\infty$, we infer 
$$
\left|\int\p_i\,(\mu_i-\mu)\right|\lesssim C\jj_\om(\mu,\p_i)^\a,
$$
and we conclude, as desired, that the left-hand side tends to $0$ as $i\to\infty$.    
\end{proof}
%
%
%
\subsection{The Dirichlet quasi-metric}\label{sec:Dirqm} 
From now on, \textbf{we assume that the orthogonality property holds} for $\om$. Recall from Examples~\ref{exam:AOK} and~\ref{exam:AONA}, that this is the case if $X$ is a compact K\"ahler manifold, and for any projective Berkovich space.

\begin{thm}\label{thm:Jmes} There exists a unique continuous functional
$$
\d_\om\colon\cM^1_\om\times\cM^1_\om\to\R_{\ge 0}, 
$$ 
such that 
\begin{equation}\label{equ:dMA}
\d_\om(\MA_\om(\f),\MA_\om(\p))=\jj_\om(\f,\p) 
\end{equation}
for all $\f,\p\in\cD_\om$. Furthermore:
\begin{itemize}
\item[(i)] for all $\mu\in\cM^1_\om$ and $\p\in\cD_\om$ 
we have 
\begin{equation}\label{equ:JmuMA}
\d_\om(\mu,\MA_\om(\p))=\jj_\om(\mu,\p),\quad\d_\om(\mu,\mu_\om)=\jj_\om(\mu);
\end{equation}
\item[(ii)] $\d_\om$ is a quasi-metric: for all $\mu,\nu,\rho\in\cM^1_\om$ we have 
\begin{equation}\label{equ:Jqmetr}
\d_\om(\mu,\nu)=0\Leftrightarrow\mu=\nu,\quad\d_\om(\mu,\nu)\approx\d_\om(\nu,\mu),\quad\d_\om(\mu,\nu)\lesssim\d_\om(\mu,\rho)+\d_\om(\rho,\nu); 
\end{equation}
\item[(iii)] the quasi-metric $\d_\om$ satisfies the H\"older continuity property
\begin{equation}\label{equ:holdJmes}
\left|\d_\om(\mu,\nu)-\d_\om(\mu',\nu')\right|\lesssim\max\{\d_\om(\mu,\mu'),\d_\om(\nu,\nu')\}^\a R^{1-\a}
\end{equation}
for all $R>0$ and $\mu,\mu',\nu,\nu'\in\cM^1_{\om,R}$, with $\a:=2^{-n}$; 
\item[(iv)] for all $R>0$ and $\f,\p\in\cD_{\om,R}$, $\mu,\nu\in\cM^1_{\om,R}$, we have the H\"older estimate
\begin{equation}\label{equ:holdmes}
\left|\int(\f-\p)(\mu-\nu)\right|\lesssim\jj_\om(\f,\p)^\a\d_\om(\mu,\nu)^{1/2} R^{1/2-\a}. 
\end{equation} 
\end{itemize}
\end{thm}
We call $\d_\om$ the \emph{Dirichlet quasi-metric} of $\cM^1_\om$. 

\begin{lem} For all $\mu\in\cM^1_{\om,R}$ and $\f,\p,\tau\in\cD_{\om,R}$, we have  
\begin{equation}\label{equ:holdJmesMA}
\left|\jj_\om(\mu,\f)-\jj_\om(\MA_\om(\tau),\p)\right|\lesssim\max\{\jj_\om(\mu,\tau),\jj_\om(\f,\p)\}^\a  R^{1-\a}; 
\end{equation}
\begin{equation}\label{equ:holdmesMA}
\left|\int(\f-\p)(\mu-\MA_\om(\tau))\right|\lesssim\jj_\om(\f,\p)^\a\jj_\om(\mu,\tau)^{1/2} R^{1/2-\a}. 
\end{equation} 
\end{lem}
\begin{proof} When $\mu$ lies in the image of $\MA_\om\colon\cD_\om\to\cM^1_\om$ this is equivalent to~\eqref{equ:holdJ} and~\eqref{equ:holdMA}, in view of~\eqref{equ:JMA} and~\eqref{equ:JMA2}. By Theorem~\ref{thm:maxcv} the image of $\MA_\om$ is dense in $\cM^1_\om$, and the general case thus follows by continuity in the strong topology of all functions of $\mu$ involved. 
\end{proof}

\begin{proof}[Proof of Theorem~\ref{thm:Jmes}] Uniqueness is clear, since~\eqref{equ:dMA} determines $\d_\om$ on the image of $\MA_\om\colon\cD_\om\to\cM^1_\om$, which is dense in the strong topology, by Theorem~\ref{thm:maxcv}. To show existence, pick $\mu,\nu\in\cM^1_\om$, and choose a maximizing sequences $(\p_i)$ for $\nu$. We can then find $R>0$ such that $\mu\in\cM^1_{\om,R}$ and $\p_i\in\cD_{\om,R}$ for all $i$, and~\eqref{equ:holdJmesMA} and~\eqref{equ:Jmutri} yield 
\begin{align*}
|\jj_\om(\mu,\p_i)-\jj_\om(\mu,\p_j)| & \lesssim \jj_\om(\p_i,\p_j)^\a R^{1-\a}\\
& \lesssim\max\{\jj_\om(\nu,\p_i),\jj_\om(\nu,\p_j)\}^\a R^{1-\a}. 
\end{align*}
This estimate implies that $(\jj_\om(\mu,\p_i))$ is a Cauchy sequence, which thus admits a limit 
\begin{equation}\label{equ:dpi}
\d_\om(\mu,\nu):=\lim_i\jj_\om(\mu,\p_i).
\end{equation}
The same estimate also shows that the limit is independent of the choice of maximizing sequence $(\p_i)$, and that the convergence in~\eqref{equ:dpi} is uniform with respect to $\mu\in\cM^1_{\om,R}$. As a consequence, $\mu\mapsto\d_\om(\mu,\nu)$ so defined is continuous on $\cM^1_\om$ for each $\nu\in\cM^1_\om$. 

By construction, \eqref{equ:JmuMA} holds, and hence also~\eqref{equ:dMA}, by~\eqref{equ:JMA}. This proves (i). 

Next, \eqref{equ:holdJmes} holds when $\nu,\nu'$ lie in the image of $\MA_\om\colon\cD_\om\to\cM^1_\om$, by applying \eqref{equ:holdJmesMA} to a maximizing sequence for $\mu'$, and the general case follows by using maximizing sequences for $\nu,\nu'$. This shows (iii), which also yields the continuity of $\d_\om$ on $\cM^1_\om\times\cM^1_\om$ (and hence conclude the proof of existence), since $\mu_i\to\mu$ strongly implies $\d_\om(\mu_i,\mu)\to \d_\om(\mu,\mu)=0$, by continuity of $\d_\om(\cdot,\mu)$. 

Similarly, (iv) follows follows by applying~\eqref{equ:holdmesMA} to a maximizing sequence $(\tau_i)$ for $\nu$.  

Finally, the first point in (ii) follows from~\eqref{equ:holdmes}, since $\cD_\om$ spans the dense subspace $\cD$ of $\Cz(X)$ (see Corollary~\ref{cor:pshspan}). By~\eqref{equ:IJsum} and~\eqref{equ:qtri}, the last two properties in (ii) hold when the measures lie in the image of $\MA_\om$, and hence in general, by continuity of $\d_\om$. 

\end{proof}

We next show: 

\begin{thm}\label{thm:M1} The quasi-metric space $(\cM^1_\om,\d_\om)$ only depends on the class $[\om]$. It is complete, and its topology coincides with the strong topology. 
\end{thm}

\begin{lem}\label{lem:weakJ} For any $\nu\in\cM^1$ and $R>0$, $\d_\om(\cdot,\nu)$ is weakly lsc on $\cM^1_{\om,R}$. 
\end{lem}
\begin{proof} When $\nu=\MA_\om(\p)$ with $\p\in\cD_\om$, \eqref{equ:JmuMA} yields $\d_\om(\cdot,\nu)=\jj_\om(\cdot,\p)$, which is weakly lsc on $\cM^1_\om$ (see~\eqref{equ:Jmuvar}). In the general case, pick a maximizing sequence $(\p_i)$ for $\nu$, and set $\nu_i:=\MA_\om(\p_i)$. By~\eqref{equ:holdJmes}, we have $\d_\om(\mu,\nu_i)\to\d_\om(\mu,\nu)$ uniformly for $\mu\in\cM^1_{\om,R}$, and the result follows. 
\end{proof}

\begin{proof}[Proof of Theorem~\ref{thm:M1}] We already know that $\cM^1_\om$ only depends on $[\om]$ (see Proposition~\ref{prop:M1ind}). Pick $\tau\in\cD_\om$, $\mu,\nu\in\cM^1_\om=\cM^1_{\om_\tau}$, and choose maximizing sequences $(\f_i)$, $(\p_i)$ in $\cD_{\om_\tau}$ for $\mu,\nu$, so that $\d_{\om_\tau}(\mu,\nu)=\lim_i\jj_{\om_\tau}(\f_i,\p_i)$. By~\eqref{equ:Jmestrans}, $(\f_i+\tau)$ and $(\p_i+\tau)$ are maximizing sequences in $\cD_{\om}$ for $\mu,\nu$, and hence $\d_\om(\mu,\nu)=\lim_i\jj_\om(\f_i,\p_i)$. Now~\eqref{equ:Jtrans} yields $\jj_{\om_\tau}(\f_i,\p_i)=\jj_\om(\f_i+\tau,\p_i+\tau)$, which proves $\d_{\om_\tau}(\mu,\nu)=\d_\om(\mu,\nu)$ and proves that $\d_\om$ only depends on $[\om]$. 

We next show that the topology of $(\cM^1_\om,\d_\om)$ is the strong topology, \ie a net $(\mu_i)$ converges strongly to $\mu\in\cM^1_\om$ iff $\d_\om(\mu_i,\mu)\to 0$. When the latter holds, \eqref{equ:holdmes} implies $\mu_i\to\mu$ weakly (since $\cD_\om$ spans the dense subspace $\cD$ of $\Cz(X)$), while~\eqref{equ:holdJmes} yields $\jj_\om(\mu_i)=\jj_\om(\mu_i,0)\to\jj_\om(\mu)$. Thus $\mu_i\to\mu$ strongly, and the converse holds by strong continuity of $\d_\om$. 

Finally, consider a Cauchy net $(\mu_i)$ in $(\cM^1,\d_\om)$. Then $\jj_\om(\mu_i)=\d_\om(\mu_i,\mu_\om)$ is eventually bounded. By weak compactness of $\cM$, we may assume, after passing to a subnet, that $(\mu_i)$ admits a weak limit $\mu\in\cM$. Since $\jj_\om$ is weakly lsc on $\cM$, we get $\jj_\om(\mu)\le\liminf_i\jj_\om(\mu_i)<+\infty$, \ie $\mu\in\cM^1$. It remains to show $\d_\om(\mu_i,\mu)\to 0$. To see this, pick $\e>0$ and $i_0$ such that $\d_\om(\mu_i,\mu_j)\le\e$ for all $i,j\ge i_0$. Since $\jj_\om(\mu_j)$ is bounded and $\mu_j\to\mu$ weakly, Lemma~\ref{lem:weakJ} yields $\d_\om(\mu_i,\mu)\le\liminf_j\d_\om(\mu_i,\mu_j)\le\e$, and we are done. 
\end{proof}

To conclude this section, we show: 
\begin{prop} For each $\nu\in\cM^1_\om$, $\d_\om(\cdot,\nu)\colon\cM^1_\om\to\R_{\ge 0}$ is strictly convex, and we further have the uniform convexity estimate
$$
(1-t)\d_\om(\mu_0,\nu)+t\d_\om(\mu_1,\nu)-\d_\om((1-t)\mu_0+t\mu_1,\nu)\gtrsim t(1-t)\d_\om(\mu_0,\mu_1). 
$$
for all $\mu_0,\mu_1\in\cM^1_\om$ and $t\in [0,1]$. 
\end{prop}
\begin{proof} By density of the image of $\MA_\om\colon\cD_\om\to\cM^1_\om$ and continuity of $\d_\om$, we may assume without loss 
$\nu=\MA_\om(\p)$ with $\p\in\cD_\om$, and hence $\d_\om(\cdot,\nu)=\jj_\om(\cdot,\p)$. Set 
$$
J_t:=(1-t)\jj_\om(\mu_0,\p)+t\jj_\om(\mu_1,\p),\quad\mu_t:=(1-t)\mu_0+t\mu_1,
$$ 
and pick $\f\in\cD_\om$. Applying~\eqref{equ:Jmu} to $\mu_0$ and $\mu_1$ yields 
$$
J_t =\en_\om(\f)-\en_\om(\p)+\int(\p-\f)\mu_t+(1-t)\jj_\om(\mu_0,\f)+t\jj_\om(\mu_1,\f), 
$$
and hence 
$$
J_t-\en_\om(\f)+\en_\om(\p)+\int(\f-\p)\mu_t\ge t(1-t)(\jj_\om(\mu_0,\p)+\jj_\om(\mu_1,\p)),  
$$
using the elementary estimate $(1-t)a+t b\ge t(1-t)(a+b)$ for $a,b\ge 0$ (see for instance~\cite[Lemma~7.29]{trivval}).
By~\eqref{equ:Jqmetr}, this implies
$$
J_t-\en_\om(\f)+\en_\om(\p)+\int(\f-\p)\mu_t\gtrsim t(1-t)\d_\om(\mu_0,\mu_1),
$$
and taking the infimum over $\f$ shows $J_t-\jj_\om(\mu_t,\p)\gtrsim t(1-t)\dd_\om(\mu_0,\mu_1)$, which concludes the proof. 
\end{proof}

%
%
\subsection{An equivalent metric on $\cM^1$}\label{sec:dR}
As recalled in~\S\ref{sec:notation}, the quasi-metric space $(\cM^1_\om,\d_\om)$ is metrizable, by general theory. Here we introduce a concrete metric that define the strong topology of $\cM^1_\om$. Recall from~\eqref{equ:DR} that 
$$
\cD_{\om,1}=\{\f\in\cD_\om\mid\jj_\om(\f)\le 1\}.
$$
\begin{prop}\label{prop:dR} Setting 
\begin{equation}\label{equ:dR}
\dd_{\om}(\mu,\nu):=\sup_{\f\in\cD_{\om,1}}\left|\int\f(\mu-\nu)\right|
\end{equation}
yields a complete metric on $\cM^1_\om$ that defines the strong topology. Furthermore:
\begin{itemize}
\item[(i)] the metric $\dd_\om$ and the Dirichlet quasi-metric $\d_\om$ share the same bounded sets;
\item[(ii)] they are H\"older equivalent on bounded sets, \ie 
\begin{equation}\label{equ:dJ}
\dd_\om(\mu,\nu)\lesssim\d_{\om}(\mu,\nu)^{1/2} R^{1/2}\quad\text{and}\quad\d_\om(\mu,\nu)\lesssim \dd_\om(\mu,\nu) R^{1/2}
\end{equation}
for all $\mu,\nu\in\cM^1_{\om,R}$ with $R\ge 1$; 
\item[(iii)] for all $\f\in\cD_\om$ and $\mu,\nu\in\cM^1_\om$ we have 
\begin{equation}\label{equ:intbound} 
\left|\int\f\,(\mu-\nu)\right|\lesssim(\jj_\om(\f)^{1/2}+1)\dd_\om(\mu,\mu). 
\end{equation}
\end{itemize}
\end{prop}

\begin{proof}[Proof of Proposition~\ref{equ:dR}] Pick $\f\in\cD_{\om}$. Since $\jj_\om(a^{-1}\f)\lesssim a^{-2}\jj_\om(\f)$ for $a\ge 1$ (see~\eqref{equ:Jquad}), we can choose $1\le a\lesssim \jj_\om(\f)^{1/2}+1$ such that $\jj_\om(a^{-1}\f)\le 1$. Then $\left|\int a^{-1}\f\,(\mu-\nu)\right|\le\dd_\om(\mu,\nu)$, which proves~\eqref{equ:intbound}. 

The first part of~\eqref{equ:dJ} is a direct consequence of~\eqref{equ:holdmes}. It shows, in particular, that $\dd_\om$ is finite valued. It is also clear that $\dd_\om$ is symmetric, vanishes on the diagonal, and satisfies the triangle inequality. Since $\cD_{\om,1}$ spans the dense subspace $\cD$ of $\Cz(X)$ (see Lemma~\ref{lem:Jspan}), $\dd_\om$ further separates points, and hence defines a metric on $\cM^1_\om$. 

The first part of~\eqref{equ:dJ} also shows that $\mu_i\to\mu$ in $\cM^1_\om$ implies $\dd_\om(\mu_i,\mu)\to 0$, by continuity of $\d_\om$, and it follows that the metric $\dd_\om$ is continuous. By density of the image of $\MA_\om$, it is thus enough to show the second half of~\eqref{equ:dJ} when $\mu=\MA_\om(\f)$ and $\nu=\MA_\om(\p)$ with $\f,\p\in\cD_\om$. Then $\jj_\om(\f)\approx\jj_\om(\mu)\le R$ and $\jj_\om(\p)\approx\jj_\om(\nu)\le R$ (see~\eqref{equ:JMA2}), while 
$$
\d_\om(\mu,\nu)=\jj_\om(\f,\p)\le\int(\f-\p)(\mu-\nu), 
$$ 
by~\eqref{equ:dMA} and~\eqref{equ:IJ}. Using~\eqref{equ:intbound}, we get the second half of~\eqref{equ:dJ}. Next pick $\mu\in\cM^1_\om$ and set $R:=\max\{1,\d_\om(\mu,\mu_\om)\}$ and $S:=\max\{1,\dd_\om(\mu,\mu_\om)\}$. Applying~\eqref{equ:dJ} to $\nu=\mu_\om$ yields
$S\lesssim R$, and also $R\lesssim S R^{1/2}$, \ie $R\lesssim S^2$. This proves (i) and (ii). Since $\d_\om$ defines the strong topology of $\cM^1_\om$ and is complete (see Theorem~\ref{thm:M1}), the same therefore holds, as desired, for $\dd_\om$. 
\end{proof}

By Theorem~\ref{thm:M1}, the quasi-metric space $(\cM^1_\om,\d_\om)$ only depends on $[\om]\in\Pos(X)$.  Here we show: 

\begin{lem}\label{lem:M1equiv} For each $\tau\in\cD_\om$ we have $\dd_{\om_\tau}\lesssim (\jj_\om(\tau)^{1/2}+1)\dd_\om$. In particular, the Lipschitz equivalence class of the metric space $(\cM^1_\om,\dd_\om)$ only depends on $[\om]$. 
\end{lem}
\begin{proof} For any $\f\in\cD_{\om_\tau}$, \eqref{equ:Jtrans} yields $\jj_{\om_\tau}(\f)=\jj_{\om_\tau}(0,\f)=\jj_\om(\tau,\f+\tau)$. 
When $\jj_{\om_\tau}(\f)\le 1$, the quasi-triangle inequality~\eqref{equ:qtri} thus yields $\jj_\om(\f+\tau)=\jj_\om(0,\f+\tau)\lesssim 1+J$  with $J:=\jj_\om(\tau)=\jj_\om(0,\tau)$. By~\eqref{equ:intbound} we infer
$$
\left|\int(\f+\tau)(\mu-\nu)\right|\lesssim(1+J^{1/2})\dd_\om(\mu,\nu),\quad\left|\int\tau(\mu-\nu)\right|\lesssim(1+J^{1/2})\dd_\om(\mu,\nu). 
$$
Thus $\left|\int\f(\mu-\nu)\right|\lesssim(1+J^{1/2})\dd_\om(\mu,\nu)$. Taking the supremum over $\f\in\cD_{\om_\tau}$ such that $\jj_{\om_\tau}(\f)\le 1$ yields the result.  
\end{proof}

%
%
%
\section{Lipschitz and H\"older estimates for the energy}\label{sec:further}
In what follows we consider $\om\in\cZ_+$ with $[\om]\in\Pos(X)$. As above, we assume that the orthogonality property holds. From now on, we further assume the submean value property (see Definition~\ref{defi:submean}), and use it to investigate the dependence of $\cM^1_\om$ on $\om$ and establish a H\"older continuity estimate for the energy pairing. 

Recall that the standing assumptions hold when $X$ is a compact connected K\"ahler manifold, or any irreducible projective Berkovich space (see Theorem~\ref{thm:submean} and Examples~\ref{exam:AOK} and~\ref{exam:AONA}). 

%
%
\subsection{Lipschitz estimates for the energy}\label{sec:lipen}
Recall from~\S\ref{sec:dR} the metric $\dd_\om$, which defines the strong topology of $\cM^1_\om$. As a first key consequence of the submean value property, we show: 

\begin{thm}\label{thm:M1ind} The Lipschitz equivalence class of the metric space $(\cM^1_\om,\dd_\om)$ is independent of $\om$. 
\end{thm} 

In particular, the topological space $\cM^1_\om$ is independent of $\om$ (see Proposition~\ref{prop:dR}), and will henceforth simply by denoted by $\cM^1$. 

\begin{lem}\label{lem:M1comm} Assume $\om'\in\cZ_+$ is commensurable to $\om$, with $[\om']\in\Pos(X)$. For all $\mu,\nu\in\cM$ we then have
\begin{equation}\label{equ:dRcomp}
\dd_{\om'}(\mu,\nu)\le e^{O(\d)}(1+T_\om)^{1/2}\dd_\om(\mu,\nu)
\end{equation}
with $\d:=\dT(\om,\om')\in\R_{\ge 0}$. 
\end{lem}

\begin{proof} Pick any $\f\in\cD_{\om'}$ such that $\jj_{\om'}(\f)=\int\f\,\mu_{\om'}-\en_{\om'}(\f)\le 1$. We need to show
\begin{equation}\label{equ:dRcompbis}
\left|\int\f(\mu-\nu)\right|\le e^{O(\d)}(1+T_\om)^{1/2}\dd_\om(\mu,\nu).
\end{equation}
By translation invariance of $\jj_\om$ on $\cD_\om$, we may assume wlog $\sup\f=0$. Then 
$$
-\en_{\om'}(\f)\le 1+T_{\om'}\le 1+ e^{O(\d)} T_\om,
$$
by~\eqref{equ:Thom}. On the other hand, 
$$
0\ge e^{(n+1)\d}(\om,e^{-\d}\f)^{n+1}=(e^\d\om,\f)^{n+1}\ge e^{O(\d)}(\om',\f)^{n+1},
$$
where the last inequality follows from Lemma~\ref{lem:ennef}, since $\om'\le e^\d\om\le e^{2\d}\om'$. Dividing by $(n+1)V_\om=e^{O(\d)}(n+1)V_{\om'}$, we get
$$
0\le\jj_{\om}(e^{-\d}\f)\le-\en_\om(e^{-\d}\f)\le -e^{O(\d)}\en_{\om'}(\f)\le e^{O(\d)}(1+T_\om).
$$
By~\eqref{equ:intbound} this implies
$$
\left|\int e^{-\d}\f(\mu-\nu)\right|\le e^{O(\d)}(1+T_\om)^{1/2}\dd_\om(\mu,\nu),
$$
and hence~\eqref{equ:dRcompbis}. 
\end{proof}

\begin{proof}[Proof of Theorem~\ref{thm:M1ind}] We argue as in the proof of Proposition~\ref{prop:submean}. On the one hand, the Lipschitz equivalence class only depends on the positive class $[\om]$, by Lemma~\ref{lem:M1equiv}. On the other hand, it only depends on the commensurability class of $\om$, by Lemma~\ref{lem:M1comm}, and we conclude since any two positive classes admit commensurable representatives, by Proposition~\ref{prop:comm}. 
\end{proof}

%
%
\subsection{Mixed Monge--Amp\`ere measures}\label{sec:mixedMA}
It will be convenient to introduce, for $\f\in\cD_\om$ and $\mu\in\cM$, the quantities
\begin{equation}\label{equ:tJ}
\tJ_\om(\f):=\jj_\om(\f)+T_\om,\quad\text{ and }\quad\tJ_\om(\mu):=\jj_\om(\mu)+T_\om. 
\end{equation}

\begin{lem}\label{lem:tJ} For each $\f\in\cD_\om$ we have 
\begin{equation}\label{equ:EtJ}
0\le\sup\f-\en_\om(\f)\le\tJ_\om(\f); 
\end{equation}
\begin{equation}\label{equ:tJapp}
\tJ_\om(\f)\approx\tJ_\om(\MA_\om(\f)). 
\end{equation}
\end{lem}
\begin{proof} By~\eqref{equ:J}, we have 
$$
\sup\f-\en_\om(\f)=\jj_\om(\f)+(\sup\f-\int\f\,\mu_\om)\le\jj_\om(\f)+T_\om.  
$$
This yields~\eqref{equ:EtJ}, while~\eqref{equ:tJapp} is a direct consequence of~\eqref{equ:JMA2}. 
\end{proof}

We next establish a key energy estimate for mixed Monge--Amp\`ere measures. 

\begin{thm}\label{thm:mixedMA} For $i=1,\dots,n$, pick $\om_i\in\cZ_+$ with $[\om_i]\in\Pos(X)$ and $\f_i\in\cD_{\om_i}$, and set
\begin{equation}\label{equ:mixed}
\mu:=([\om_1]\inter[\om_n])^{-1}(\om_1+\ddc\f_1)\winter(\om_n+\ddc\f_n)\in\cM. 
\end{equation}
Then $\mu$ lies in $\cM^1$, and satisfies: 
\begin{itemize}
\item[(i)] if each $\om_i$ is commensurable to $\om$, then 
$$
\tJ_\om(\mu)\lesssim e^{O(\d)}\max_i\tJ_{\om_i}(\f_i)
$$
with $\d:=\max_i\dT(\om_i,\om)$; 
\item[(ii)] in the general case, 
$$
\jj_\om(\mu)\lesssim C(\max_i\jj_{\om_i}(\f_i)+1), 
$$
where $C>0$ only depends on $\om$ and the $\om_i$. 
\end{itemize}
\end{thm}
\begin{proof} Assume first that each $\om_i$ is commensurable to $\om$. Set 
$$
J:=\max_i\tJ_{\om_i}(\f_i),\quad V:=V_\om,\quad V_i:=V_{\om_i},
$$
and observe that
\begin{equation}\label{equ:volcomp}
e^{-n\d} V\le V_i\le e^{n\d}V,\quad e^{-n\d} V\le [\om_1]\inter[\om_n]\le e^{n\d} V.
\end{equation}
Since $\mu$ is unchanged when the $\f_i$'s are translated by constants, we may assume without loss that $\sup\f_i=0$. Then
$$
0\ge(\om_i,\f_i)^{n+1}=(n+1)V_i\en_{\om_i}(\f_i)\ge-(n+1)V_i\tJ_{\om_i}(\f_i),
$$
by~\eqref{equ:EtJ}, and hence  
\begin{equation}\label{equ:enbelow1}
0\ge(\om_i,\f_i)^{n+1}\gtrsim -e^{O(\d)}V J, 
\end{equation}
using~\eqref{equ:volcomp}. Now pick $\p\in\cD_\om$ such that $\sup\p=0$, and set $R:=\jj_\om(\p)$. On the one hand, 
\begin{equation}\label{equ:enbelow2}
0\ge (\om,\p)^{n+1}=(n+1)V\en_\om(\p)=-(n+1)V R. 
\end{equation}
On the other hand, using~\eqref{equ:volcomp} again, we have 
\begin{align*}
0 & \ge e^{-n\d} V\int\p\,\mu\ge\int\p\,(\om_1+\ddc\f_1)\winter(\om_n+\ddc\f_n)\\
& =(\om,\p)\cdot(\om_1,\f_1)\inter(\om_n,\f_n)-(\om,0)\cdot(\om_1,\f_1)\inter(\om_n,\f_n)\\
& \ge (\om,\p)\cdot(\om_1,\f_1)\inter(\om_n,\f_n)\gtrsim e^{O(\d)}\min\{(\om,\p)^{n+1},\min_i (\om_i,\f_i)^{n+1}\}, 
\end{align*}
by~\eqref{equ:enmon2} and Theorem~\ref{thm:mixedenbound}. Combined with~\eqref{equ:enbelow1} and~\eqref{equ:enbelow2}, this yields 
$$
\int\p(\mu_\om-\mu)\le -\int\p\,\mu\lesssim e^{O(\d)}(R+J). 
$$
By Lemma~\ref{lem:finen}, we infer, 
$$
\jj_\om(\mu)\lesssim e^{O(\d)}\inf_{R>0}(R+J)\left(1+R^{-1}(R+J)\right)\le e^{O(\d)}J,
$$
which concludes the proof of (i) (using~\eqref{equ:Thom}). 

We now consider the general case. By Proposition~\ref{prop:comm}, we can choose $\tau\in\cD_\om$ and $\tau_i\in\cD_{\om_i}$ such that $\om':=\om_{\tau}$ and $\om'_i:=\om_{i,\tau_i}$ are commensurable for all $i$. Then 
$$
\mu=([\om'_1]\inter[\om'_n])^{-1}(\om'_1+\ddc\f'_1)\winter(\om'_n+\ddc\f_n)
$$ 
with $\f'_i:=\f_i-\tau_i$, and hence
$$
\jj_{\om'}(\mu)\lesssim\max_i(\jj_{\om'_i}(\f'_i)+T_{\om'_i}), 
$$
by Theorem~\ref{thm:mixedMA}. By~\eqref{equ:Jmestrans2}, we have $\jj_\om(\mu)\le\jj_{\om'}(\mu)+C$ and $\jj_{\om'_i}(\f'_i)\le\jj_\om(\f_i)+C$ with $C>0$ independent of $\mu$, and (ii) follows. 
\end{proof}

As a consequence, we get the following Lipschitz estimate for the energy:

\begin{cor}\label{cor:lipen} Pick $\om'\in\cZ_+$ such that $[\om']\in\Pos(X)$, and $\mu\in\cM^1$. 
\begin{itemize}
\item[(i)] If $\om'$ is commensurable to $\om$, then
$$
\tJ_{\om'}(\mu)\approx e^{O(\d)}\tJ_\om(\mu)
$$
with $\d:=\dT(\om,\om')$. 
\item[(ii)] In the general case, there exists $C>0$ only depending on $\om,\om'$ such that 
$$
\jj_{\om'}(\mu)\le C(\jj_\om(\mu)+1). 
$$
\end{itemize}
\end{cor}
We refer to~\eqref{equ:tJvar} below for a more precise estimate when $\d$ is small. 

\begin{proof} Assume first $\om'$ commensurable. Pick a maximizing sequence $(\f_j)$ for $\mu$ in $\cD_\om$, and set $\mu_j:=\MA_\om(\f_j)$. Then $\mu_j\to\mu$ strongly in $\cM^1$ (see Theorem~\ref{thm:maxcv}), and hence $\jj_\om(\mu_j)\to\jj_\om(\mu)$ and $\jj_{\om'}(\mu_j)\to\jj_{\om'}(\mu)$. For each $j$ we have $\jj_\om(\f_j)\approx\jj_\om(\mu_j)$ (see~\eqref{equ:JMA2}). Theorem~\ref{thm:mixedMA} thus yields 
$\tJ_{\om'}(\mu_j)\lesssim e^{O(\d)}\tJ_\om(\mu_j)$, and (i) follows. 

In the general case, we can choose $\tau\in\cD_\om$ and $\tau'\in\cD_{\om'}$ such that $\om_\tau$ and $\om_{\tau'}$ are commensurable (see Proposition~\ref{prop:comm}). By~\eqref{equ:Jmestrans2}, we then have $\jj_{\om'}(\mu)\le\jj_{\om_\tau}(\mu)+C$ and $\jj_{\om_{\tau'}}(\mu)\le\jj_{\om'}(\mu)+C$ with $C>0$ independent of $\mu$, and (ii) now follows from (i). 
\end{proof}

%
%
\subsection{H\"older continuity of the energy pairing}
Recall from~\S\ref{sec:pos} that 
$$
\cZ_\om=\left\{\theta\in\cZ\mid\|\theta\|_\om<\infty\right\}.
$$
Using the above results, we establish a general H\"older continuity property for the energy pairing. 

\begin{thm}\label{thm:enhold} For $i=0,\dots,n$, pick $\theta_i,\theta'_i\in\cZ_\om$ and $\f_i,\f'_i\in\cD_\om$, normalized by $\int\f_i\,\mu_\om=\int\f'_i\,\mu_\om=0$. Then
$$
\left|(\theta_0,\f_0)\inter(\theta_n,\f_n)-(\theta'_0,\f'_0)\inter(\theta'_n,\f'_n)\right|\lesssim A\left(\max_i\|\theta_i-\theta'_i\|_\om J+\max_i\jj_\om(\f_i,\f'_i)^\a J^{1-\a}\right). 
$$
with $\a:=2^{-n}$ and 
$$
A:=V_\om\prod_i\left(1+\|\theta_i\|_\om+\|\theta'_i\|_\om\right),\quad J:=\max_i\tJ_\om(\f_i). 
$$
In particular, 
\begin{equation}\label{equ:enbd}
\left|(\theta_1,\f_1)\inter(\theta_n,\f_n)\right|\lesssim V_\om\prod_i\left(1+\|\theta_i\|_\om\right)\max_i\tJ_\om(\f_i). 
\end{equation}
\end{thm} 

\begin{proof} Assume first $\theta_i=\theta'_i$ for all $i$. By symmetry of the energy pairing, we may assume $\f_i=\f'_i$ for $i\ge 1$. For $i=1,\dots,n$ set $t_i:=1+\|\theta_i\|_\om$. Then $\|t_i^{-1}\theta_i\|_\om\le 1$, and 
$$
\jj_\om(t_i^{-1}\f_i)\le t_i^{-1}\jj_\om(\f_i)\le\jj_\om(\f_i), 
$$
by convexity of $\jj_\om$ on $\cD_\om$. By homogeneity, we may thus assume $\|\theta_i\|_\om\le 1$ for all $i=1,\dots,n$. Thus $-\om\le\theta_i\le\om$, and hence $\theta_i=\theta_i^+-\theta_i^-$ where $\theta_i^+:=\theta_i+2\om$ and $\theta_i^-:=2\om$ both satisfy $\om\le\theta_i^\pm\le 3\om$. By multilinearity, we may finally assume $\om\le\theta_i\le 3\om$. By Proposition~\ref{prop:enpairing}~(i), it then suffices to show 
\begin{equation}\label{equ:mixedMA}
\left|\int(\f_0-\f'_0)(\theta_1+\ddc\f_1)\winter(\theta_n+\ddc\f_n)\right|\le V_\om\jj_\om(\f_0,\f'_0)^\a J^{1-\a}.
\end{equation}
Since $\om\le\theta_i\le 3\om$, Theorem~\ref{thm:mixedMA} shows that 
$$
\mu:=([\theta_0]\inter[\theta_n])^{-1}(\theta_1+\ddc\f_1)\winter(\theta_n+\ddc\f_n)\in\cM
$$
satisfies $\jj_\om(\mu)\lesssim J$. By~\eqref{equ:holdmes}, we infer
$$
\left|\int(\f_0-\p_0)(\mu-\mu_\om)\right|\lesssim  \jj_\om(\f_0,\p_0)^\a J^{1-\a},
$$
which yields~\eqref{equ:mixedMA} since $\int(\f_0-\p_0)\mu_\om=0$ and $[\theta_0]\inter[\theta_n]\lesssim V_\om$.

In the general case, we may again assume $(\theta_i,\f_i)=(\theta'_i,\f'_i)$ for $i\ge 1$, by symmetry and multilinearity of the energy pairing. Then
$$
(\theta_0,\f_0)\cdot(\theta_1,\f_1)\inter(\theta_n,\f_n)-(\theta'_0,\f'_0)\cdot(\theta_1,\f_1)\inter(\theta_n,\f_n)
$$
$$
=(\theta_0-\theta'_0,0)\cdot (\theta_1,\f_1)\inter(\theta_n,\f_n)+(0,\f_0-\f'_0)\cdot(\theta_1,\f_1)\inter(\theta_n,\f_n), 
$$
where the last term has already been estimated by~\eqref{equ:mixedMA}. We are thus reduced to showing
$$
\left|(\theta_0,0)\cdot (\theta_1,\f_1)\inter(\theta_n,\f_n)\right|\lesssim A J\|\theta_0\|_\om. 
$$
By homogeneity we may further assume $\|\theta_0\|_\om=1$, and the desired estimate now follows from the first step of the proof applied to $\f'_i=0$, using $(\theta_0,0)\inter(\theta_n,0)=0$. 
\end{proof}

%
%
%
\section{Twisted energy and differentiability}\label{sec:twisteddiff} 
As in~\S\ref{sec:further}, we assume that the orthogonality and submean properties hold, and recall that this is satisfied when $X$ is a compact connected K\"ahler manifold or any irreducible projective Berkovich space. We fix $\om\in\cZ_+$ with $[\om]\in\Pos(X)$. In this section we introduce and study the twisted energy of a measure, and show that it computes certain directional derivatives of the energy. 
%
%
\subsection{The twisted energy of a measure}\label{sec:twisted}

In view of~\eqref{equ:Eom}, the directional derivative 
\begin{equation}\label{equ:twEder}
\nabla_\theta\en_\om(\f):=\frac{d}{dt}\bigg|_{t=0}\en_{\om+t\theta}(\f)
\end{equation}
is well-defined for any $\theta\in\cZ$ and $\f\in\cD$, and given by  
\begin{equation}\label{equ:nablaE}
\nabla_\theta\en_\om(\f)=\en_\om^\theta(\f)-V_\om^\theta\en_\om(\f), 
\end{equation}
where 
\begin{equation}\label{equ:enomth}
\en_\om^\theta(\f):=V_\om^{-1}(\theta,0)\cdot(\om,\f)^n
\end{equation}
and 
$$
V_\om^\theta:=n V_\om^{-1}[\theta]\cdot[\om]^{n-1}. 
$$
Note that $\en_\om^\theta(\f)$ is a linear function of $\theta\in\cZ$, and 
\begin{equation}\label{equ:twistedtrans}
\en_\om^\theta(\f+c)=\en_\om^\theta(\f)+c V_\om^\theta\text{ for }c\in\R,
\end{equation}
while $\nabla_\theta\en_\om$ is translation invariant. For all $\f,\p\in\cD$, we further have 
\begin{equation}\label{equ:twen}
\en_\om^\theta(\f)-\en_\om^\theta(\p)=\sum_{j=0}^{n-1}V_\om^{-1}\int(\f-\p)\theta\wedge\om_\f^j\wedge\om_\p^{n-1-j}. 
\end{equation}

\begin{exam} By~\eqref{equ:Eom}, we have $\en_{(1+t)\om}((1+t)\f)=(1+t)\en_\om(\f)$. By~\eqref{equ:deren}, this implies 
$\nabla_\om\en_\om(\f)+\int\f\,\MA_\om(\f)=\en_\om(\f)$, and hence 
\begin{equation}\label{equ:nablaom}
\nabla_\om\en_\om(\f)=\en_\om(\f)-\int\f\,\MA_\om(\f)=\jj_\om(\f,0)\approx\jj_\om(0,\f)=\jj_\om(\f). 
\end{equation}
\end{exam}

\begin{exam} For each $\p\in\cD$ and $t\in\R$, \eqref{equ:entrans} yields $\en_{\om+t\ddc\p}(\f)=\en_\om(\f+t\p)-\en_\om(t\p)$, and using again~\eqref{equ:deren} we get
\begin{equation}\label{equ:nabladdc}
\nabla_{\ddc\p}\en_\om(\f)=\int\p(\MA_\om(\f)-\mu_\om).
\end{equation}
\end{exam}

\begin{exam} For any $\tau\in\cD_\om$, we similarly have $\en_{\om_\tau+t\theta}(\f)=\en_{\om+t\theta}(\f+\tau)-\en_{\om+t\theta}(\tau)$ and  
\begin{equation}\label{equ:nablatrans}
\nabla_\theta\en_{\om_\tau}(\f)=\nabla_\theta\en_\om(\f+\tau)-\nabla_\theta\en_\om(\tau).
\end{equation}
\end{exam}

\begin{lem}\label{lem:nablaholder} For all $\theta\in\cZ$ and $\f,\p\in\cD_\om$, the following holds: 
\begin{itemize}
\item[(i)] if $\theta\in\cZ_\om$, then
\begin{equation}\label{equ:nablaEhold}
|\nabla_\theta\en_\om(\f)-\nabla_\theta\en_\om(\p)|\lesssim\jj_\om(\f,\p)^{\a}\max\{\tJ_\om(\f),\tJ_\om(\p)\}^{1-\a}\|\theta\|_\om
\end{equation}
with $\a:=2^{-n}$; 
\item[(ii)] in the general case, there exist $C>0$ only depending on $\om$ and $\theta$ such that 
\begin{equation}\label{equ:nablaEhold2}
|\nabla_\theta\en_\om(\f)-\nabla_\theta\en_\om(\p)|\le C\jj_\om(\f,\p)^{\a} (\max\{\jj_\om(\f),\jj_\om(\p)\}+1)^{1-\a}. 
\end{equation}
\end{itemize}
\end{lem}
\begin{proof} By translation invariance of $\nabla_\theta\en_\om$, we may assume $\int\f\,\mu_\om=\int\p\,\mu_\om=0$. 
When $\theta$ lies in $\cZ_\om$, Theorem~\ref{thm:enhold} applied to~\eqref{equ:Eom} and~\eqref{equ:enomth} yields 
$$
\left|\en_\om(\f)-\en_\om(\p)\right|\le C\quad\text{and}\quad\left|\en_\om^\theta(\f)-\en_\om^\theta(\p)\right|\le C(1+\|\theta\|_\om)
$$
with $C:=\jj_\om(\f,\p)^{\a}\max\{\tJ_\om(\f),\tJ_\om(\p)\}^{1-\a}$. By homogeneity of $\en^\theta_\om$ with respect to $\theta$, we may replace $1+\|\theta\|_\om$ with $\|\theta\|_\om$ in the last estimate, and (i) now follows from~\eqref{equ:nablaE} together with $|V_\om^\theta|\le n\|\theta\|_\om$. 

In the general case, pick $\tau\in\cD_\om$ such that $\theta\in\cZ_{\om_\tau}$ (see Proposition~\ref{prop:normfin}). Given $\f,\p\in\cD_\om$, we then have $\f-\tau,\p-\tau\in\cD_{\om_\tau}$, and~\eqref{equ:nablatrans}, \eqref{equ:Jtrans} yield 
$$
\nabla_\theta\en_{\om_\tau}(\f-\tau)-\nabla_\theta\en_{\om_\tau}(\p-\tau)=\nabla_\theta\en_\om(\f)-\nabla_\theta\en_\om(\p),\quad\jj_{\om_\tau}(\f-\tau,\p-\tau)=\jj_\om(\f,\p),
$$
$$
\jj_{\om_\tau}(\f-\tau)=\jj_{\om_\tau}(0,\f-\tau)=\jj_\om(\tau,\f)\lesssim\jj_\om(\f)+\jj_\om(\tau),\quad \jj_{\om_\tau}(\p-\tau)\lesssim\jj_\om(\p)+\jj_\om(\tau). 
$$
By (i) we thus get
$$
|\nabla_\theta\en_\om(\f)-\nabla_\theta\en_\om(\p)|\lesssim\jj_\om(\f,\p)^{\a}\left(\max\{\tJ_\om(\f),\tJ_\om(\p)\}+\tJ_\om(\tau)\right)^{1-\a}\|\theta\|_{\om_\tau}, 
$$
which proves (ii). 
\end{proof}

\begin{prop}\label{prop:nablaJ} For any $\theta\in\cZ$, there exists a unique strongly continuous functional 
$\jj_\om^\theta\colon\cM^1\to\R$, the \emph{$\theta$-twisted energy}, such that 
\begin{equation}\label{equ:nablaJ}
\jj_\om^\theta(\MA_\om(\f))=\nabla_\theta\en_\om(\f)
\end{equation}
for all $\f\in\cD_\om$. For all $\mu,\nu\in\cM^1$, we also have: 
\begin{itemize}
\item[(i)] $\jj_\om^\theta(\mu)$ is a linear function of $\theta$; 
\item[(ii)] if $\theta\in\cZ_\om$, then
\begin{equation}\label{equ:nablaJhold}
|\jj_\om^\theta(\mu)-\jj_\om^\theta(\nu)|\lesssim\d_\om(\mu,\nu)^{\a} \max\{\tJ_\om(\mu),\tJ_\om(\nu)\}^{1-\a}\|\theta\|_\om; 
\end{equation}
with $\a:=2^{-n}$, and hence
\begin{equation}\label{equ:twistedJbd}
|\jj_\om^\theta(\mu)|\lesssim\tJ_\om(\mu)\|\theta\|_\om;
\end{equation}
 \item[(iii)] in the general case, there exist $C>0$ only depending on $\om$ and $\theta$ such that 
\begin{equation}\label{equ:nablaJhold2}
|\jj_\om^\theta(\mu)-\jj_\om^\theta(\nu)|\le C\d_\om(\mu,\nu)^{\a} (\max\{\jj_\om(\mu),\jj_\om(\nu)\}+1)^{1-\a},
\end{equation}
and hence
\begin{equation}\label{equ:twistedJbd2}
|\jj_\om^\theta(\mu)|\le C(\jj_\om(\mu)+1). 
\end{equation}
\end{itemize}
\end{prop}

\begin{proof} Assume $\f,\p\in\cD_\om$ satisfy $\MA_\om(\f)=\MA_\om(\p)$. Then $\jj_\om(\f,\p)=0$, and hence $\nabla_\theta\en_\om(\f)=\nabla_\theta\en_\om(\p)$, by~\eqref{equ:nablaEhold2}. As a result, there exists a unique function $\jj_\om^\theta$ on the image of $\MA_\om\colon\cD_\om\to\cM^1$ such that \eqref{equ:nablaJ} is satisfied. For all $\mu,\nu\in\cM^1$ in the image of $\MA_\om$, it further follows from~\eqref{equ:nablaEhold2} that~\eqref{equ:nablaJhold2} holds. This shows that $\jj_\om^\theta$ is uniformly continuous on a dense subspace of the quasi-metric space $(\cM^1,\d_\om)$, and hence admits a unique continuous extension $\jj_\om^\theta\colon\cM^1\to\R$. Finally, (i) holds by linearity of $\nabla_\theta\en_\om$ with respect to $\theta$,(ii) and (iii) follow, by continuity, from Lemma~\ref{lem:nablaholder}.
\end{proof}

Using~\eqref{equ:nablaJ} (with $\f=0$), \eqref{equ:nablaom}, \eqref{equ:nabladdc} and~\eqref{equ:nablatrans}, we further get, for all $\mu\in\cM^1$, $\p\in\cD$ and $\tau\in\cD_\om$, 
\begin{equation}\label{equ:Jomzero}
\jj_\om^\theta(\mu_\om)=0; 
\end{equation}
\begin{equation}\label{equ:Jomom}
\jj_\om^\om(\mu)=\jj_\om(\mu); 
\end{equation}
\begin{equation}\label{equ:Jomddc}
\jj_\om^{\ddc\p}(\mu)=\int\p(\mu-\mu_\om); 
\end{equation}
\begin{equation}\label{equ:Jomtau}
\jj_{\om_\tau}^{\theta}(\mu)=\jj_\om^\theta(\mu)+c
\end{equation}
with $c\in\R$ uniquely determined by~\eqref{equ:Jomzero}, \ie $c=\jj_{\om_\tau}^\theta(\mu_\om)=-\jj_\om^\theta(\mu_{\om_\tau})$. 

%
%
\subsection{H\"older continuity of the twisted energy}\label{sec:holdertwist}
The following estimates will be the key ingredients for the continuity of coercivity thresholds (see Theorem~\ref{thm:threshcont} below).

\begin{thm}\label{thm:twistedhold} Pick $\om,\om'\in\cZ_+$ with $\d:=\dT(\om,\om')\le 1$ and $[\om']\in\Pos(X)$. For all $\theta,\theta'\in\cZ_\om$ and $\mu\in\cM^1$, we then have 
\begin{equation}\label{equ:twistedlip}
\left|\jj_{\om}^{\theta}(\mu)-\jj_{\om'}^{\theta'}(\mu)\right|\lesssim\left(\d^a\|\theta\|_\om+\|\theta-\theta'\|_\om\right)\tJ_\om(\mu)
\end{equation}
and 
\begin{equation}\label{equ:tJvar} 
\tJ_{\om}(\mu)=(1+O(\d^\a))\tJ_{\om'}(\mu).
\end{equation}
\end{thm}

\begin{lem}\label{lem:Lipen} Assume $\om\le\om'\le e^{\d}\om$ with $\d\in [0,2]$. Pick $\f\in\cD_\om\subset\cD_{\om'}$, and set $\mu:=\MA_\om(\f)$, $\mu':=\MA_{\om'}(\f)$. Then:
\begin{equation}\label{equ:Lipennabla}
|\nabla_\theta\en_\om(\f)-\nabla_\theta\en_{\om'}(\f)|\lesssim \d \tJ_\om(\mu)\|\theta\|_\om;
\end{equation}
\begin{equation}\label{equ:maxJ}
\max\{\tJ_{\om'}(\mu),\tJ_{\om'}(\mu')\}\lesssim \tJ_\om(\mu); 
\end{equation}
\begin{equation}\label{equ:distlip} 
\jj_{\om'}(\mu,\mu')\lesssim\d \tJ_\om(\mu).
\end{equation}
\end{lem}

\begin{proof} All three estimates are invariant under translation of $\f$ by a constant, and we shall rely on a different normalization for each of them. We first normalize $\f$ by $\int\f\,\mu_\om=0$. By homogeneity, we may further assume $\|\theta\|_\om=1$. Since 
$$
\en_{\om}^\theta(\f)-\en_{\om'}^\theta(\f)=V_{\om}^{-1}(\theta,0)\cdot(\om,\f)^n-V_{\om'}^{-1}(\theta,0)\cdot(\om',\f)^n
$$
with $V_{\om'}/V_\om=1+O(\d)$, Theorem~\ref{thm:enhold} yields
$$
\left|\en_{\om}^\theta(\f)-\en_{\om'}^\theta(\f)\right|\lesssim\|\om'-\om\|_\om \tJ_\om(\mu)\lesssim\d \tJ_\om(\mu). 
$$
We similarly get $\left|\en_{\om}(\f)-\en_{\om'}(\f)\right|\lesssim\d\tJ_\om(\mu)$, and~\eqref{equ:Lipennabla} follows, using the trivial estimate
$$V_{\om'}^\theta=V_\om^\theta+O(\d).
$$
 Next, we normalize $\f$ by $\sup\f=0$. By Lemma~\ref{lem:ennef}, we then have 
$$
\jj_{\om'}(\mu')\approx\jj_{\om'}(\f)=\int\f\,\mu_{\om'}-\en_{\om'}(\f)\le-\en_{\om'}(\f)\lesssim -\en_{\om}(\f)\le \tJ_\om(\mu). 
$$
On the other hand, Corollary~\ref{cor:lipen} yields $\tJ_{\om'}(\mu)\lesssim \tJ_\om(\mu)$, and~\eqref{equ:maxJ} follows.

Finally, we normalize $\f$ by $\int\f\,\mu_{\om'}=0$. Pick a maximizing sequence $(\p_i)$ in $\cD_{\om'}$ for $\mu$, also normalized by $\int\p_i\,\mu_{\om'}=0$, and set $\mu_i:=\MA_{\om'}(\p_i)$. By Theorem~\ref{thm:maxcv}, $\jj_{\om'}(\mu,\mu')$ is the limit of 
$$
\jj_{\om'}(\mu_i,\mu')=\jj_{\om'}(\p_i,\f)\approx\int(\p_i-\f)(\mu'-\mu_i)=\int(\p_i-\f)(\mu'-\mu)+o(1),
$$
where the last two points hold by~\eqref{equ:IJ} and~\eqref{equ:holdmes}, respectively. 
Now 
\begin{align*}
\int(\p_i-\f)(\mu'-\mu) & =\left[V_{\om'}^{-1}(0,\p_i)\cdot(\om',\f)^n-V_\om^{-1}(0,\p_i)\cdot (\om,\f)^n\right]\\
& -\left[V_{\om'}^{-1}(0,\f)\cdot(\om',\f)^n-V_\om^{-1}(0,\f)\cdot (\om,\f)^n\right],
\end{align*} 
where 
$$
\left|V_{\om'}^{-1}(0,\p_i)\cdot(\om',\f)^n-V_\om^{-1}(0,\p_i)\cdot (\om,\f)^n\right|\lesssim\|\om'-\om\|_{\om'}\max\{\tJ_{\om'}(\f),\tJ_{\om'}(\p_i)\}
$$
and 
$$
\left|V_{\om'}^{-1}(0,\f)\cdot(\om',\f)^n-V_\om^{-1}(0,\f)\cdot (\om,\f)^n\right|\lesssim\|\om'-\om\|_{\om'}\tJ_{\om'}(\f), 
$$
by Theorem~\ref{thm:enhold}. Since $\|\om'-\om\|_{\om'}\lesssim\d$, $\jj_{\om'}(\f)\approx\jj_{\om'}(\mu')$ and $\jj_{\om'}(\p_i)\approx\jj_{\om'}(\mu_i)\to\jj_{\om'}(\mu)$, this proves~\eqref{equ:distlip}, thanks to~\eqref{equ:maxJ}. 
\end{proof}
 
\begin{proof}[Proof of Theorem~\ref{thm:twistedhold}] Note first that $\om'':=e^{-\d}\om\le\om'\le e^\d\om$, and hence
$$
\om''\le\om'\le e^{2\d}\om'',\quad\om''\le\om\le e^\d\om''.
$$
Arguing successively with $\om'',\om'$, and with $\om'',\om$, and relying on Corollary~\ref{cor:lipen}, it is thus enough to prove the result when $\om\le\om'\le e^\d\om$, which we henceforth assume. By density of the image of $\MA_\om\colon\cD_\om\to\cM^1$ and strong continuity of $\jj_\om^\theta$ and $\jj_{\om'}^\theta$ (recall that the strong topology of $\cM^1$ is independent of $\om$), we may assume $\mu=\MA_\om(\f)$ with $\f\in\cD_\om\subset\cD_{\om'}$. As in Lemma~\ref{lem:Lipen}, set 
$\mu':=\MA_{\om'}(\f)$. By~\eqref{equ:nablaJ}, we have 
$$
\jj_\om^\theta(\mu)=\nabla_\theta\en_\om(\f),\quad\jj_{\om'}^{\theta}(\mu')=\nabla_\theta\en_{\om'}(\f),
$$
and~\eqref{equ:Lipennabla} thus yields
\begin{equation}\label{equ:twJhold}
\left|\jj_\om^\theta(\mu)-\jj_{\om'}^{\theta}(\mu')\right|\lesssim\d \tJ_\om(\mu)\|\theta\|_\om. 
\end{equation}
On the other hand, \eqref{equ:nablaJhold}, implies
\begin{equation}\label{equ:twJhold2}
\left|\jj_{\om'}^\theta(\mu)-\jj_{\om'}^\theta(\mu')\right|\lesssim\d^\a \tJ_\om(\mu)\|\theta\|_\om,
\end{equation}
thanks to~\eqref{equ:maxJ} and~\eqref{equ:distlip}. Finally, \eqref{equ:twistedJbd} yields
\begin{equation}\label{equ:twJhold3} 
\left|\jj_{\om'}^\theta(\mu)-\jj_{\om'}^{\theta'}(\mu)\right|=\left|\jj_{\om'}^{\theta-\theta'}(\mu)\right|\lesssim\tJ_{\om'}(\mu)\|\theta-\theta'\|_{\om'},
\end{equation}
and summing up~\eqref{equ:twJhold}, \eqref{equ:twJhold2} and~\eqref{equ:twJhold3} yields~\eqref{equ:twistedlip}. Applying the latter estimate with $\theta:=\om$, $\theta':=\om'$, we get 
\begin{equation}\label{equ:Jvar}
\left|\jj_{\om}(\mu)-\jj_{\om'}(\mu)\right|\lesssim\d^\a\tJ_\om(\mu),
\end{equation}
in view of~\eqref{equ:Jomom}. Since we also have $T_{\om'}=(1+O(\d))T_\om$ (see~\eqref{equ:Thom}), \eqref{equ:tJvar} follows. 
\end{proof}

%
%
%
\subsection{Differentiability of the energy}

\begin{thm}\label{thm:diffen} For any $\theta\in\cZ$ and $\mu\in\cM^1$, we have 
$$
\frac{d}{dt}\bigg|_{t=0}\jj_{\om+t\theta}(\mu)=\jj_\om^\theta(\mu)
$$
\end{thm}
If we do not assume $\theta\in\cZ_\om$, the condition $\om+t\theta\ge 0$ will fail in general, but one can still make sense of $\jj_{\om+t\theta}(\mu)$, see Remark~\ref{rmk:Jnonpos}.

\begin{lem}\label{lem:entwbd} Pick $\theta\in\cZ_\om$ and $\f\in\cD_\om$ normalized by $\int\f\,\mu_\om=0$.Then
$$
\left|\en_\om(\f)\right|\lesssim \tJ_\om(\f),\quad\left|\int\f\,\MA_\om(\f)\right|\lesssim \tJ_\om(\f),\quad \left|\en_\om^\theta(\f)\right|\lesssim\|\theta\|_\om \tJ_\om(\f).
$$
\end{lem}
\begin{proof} We have 
$$
V_\om(n+1)\en_\om(\f)=(\om,\f)^{n+1},\quad V_\om\int\f\,\MA_\om(\f)=(0,\f)\cdot(\om,\f)^n,\quad V_\om\en_\om^\theta(\f)=(\theta,0)\cdot (\om,\f)^n,
$$
and the estimates thus follow from~\eqref{equ:enbd} (and homogeneity in $\theta$). 
\end{proof}

\begin{lem}\label{lem:diffen} Pick $\theta\in\cZ_\om$ such that $\|\theta\|_\om<1$ and $[\om+\theta]\in\Pos(X)$. For all $\mu\in\cM^1$ we then have 
$$
\jj_{\om+\theta}(\mu)\ge \jj_\om(\mu)+\jj_\om^\theta(\mu)-O(\tJ_\om(\mu)\|\theta\|_\om^2). 
$$
\end{lem}
Recall that, in this paper, the implicit constant in $O$ only depends on $n$. 

\begin{proof} Set $J:=\tJ_\om(\mu)$ and $\e:=\|\theta\|_\om$. If $\e=0$, then $\theta=0$ and the result is clear. We may thus assume $\e>0$ and write $\theta=\e\tilde\theta$ with $\|\tilde\theta\|_\om=1$ and $\e\in(0,1)$. 

By density of the image of $\MA_\om\colon\cD_\om\to\cM^1$, it is enough to prove the result when $\mu=\MA_\om(\f)$ with $\f\in\cD_\om$, which we normalize by $\int\f\,\mu_\om=0$. By~\eqref{equ:JMA}, \eqref{equ:nablaJ} and~\eqref{equ:nablaE}, we then have 
\begin{equation}\label{equ:ennab}
\jj_\om(\mu)=\en_\om(\f)-\int\f\,\mu,\quad\jj_\om^\theta(\mu)=\en_\om^\theta(\f)-V_\om^\theta\en_\om(\f).
\end{equation}
Note that 
\begin{equation}\label{equ:compom}
(1-\e)\om\le\om+\theta\le(1+\e)\om\le(1-\e)^{-1}\om,
\end{equation}
and hence 
$$
(1-\e)\f\in\cD_{(1-\e)\om}\subset\cD_{\om+\theta}. 
$$ 
By~\eqref{equ:envar}, this yields 
\begin{equation}\label{equ:Jbelow}
\jj_{\om+\theta}(\mu)\ge\en_{\om+\theta}\left((1-\e)\f\right)-(1-\e)\int\f\,\mu. 
\end{equation}
\begin{align*}
(n+1)V_{\om+\theta}\en_{\om+\theta}\left((1-\e)\f)\right) & =(\om+\theta,(1-\e)\f)^{n+1}=\left((\om,\f)+(\theta,-\e\f)\right)^{n+1}\\
&=(\om,\f)^{n+1}+(n+1)\left[(\theta,0)\cdot (\om,\f)^n-\e(0,\f)\cdot(\om,\f)^n\right]+\e^2 a(\e)\\
& =(n+1)V_\om\left[\en_\om(\f)+\en_\om^\theta(\f)-\e\int\f\,\mu\right]+\e^2 a(\e)
\end{align*}
with 
$$
a(\e):=\sum_{j=2}^{n+1}{n+1\choose j}\e^{j-2}(\tilde\theta,-\f)^j\cdot (\om,\f)^{n+1-j}. 
$$
Since $\|\tilde\theta\|_\om=1$, \eqref{equ:enbd} yields $|a(\e)|\lesssim V_\om J$. Combining this with 
$$
V_{\om+\theta}=[\om+\theta]^n=[\om]^n+n[\theta]\cdot[\om]^{n-1}+O(\e^2[\om]^n)=V_\om\left(1+V_\om^\theta+O(\e^2)\right), 
$$
and
$$
\left|V_\om^\theta\right|\lesssim\e,\quad\left|\en_\om(\f)\right|\lesssim J,\quad\left|\int\f\,\mu\right|\le J,\quad \left|\en_\om^\theta(\f)\right|\lesssim\e J
$$
(see Lemma~\ref{lem:entwbd}), we infer 
\begin{align*}
\en_{\om+\theta}((1-\e)\f) & \ge \left(1-V_\om^\theta+O(\e^2)\right)\left(\en_\om(\f)+\en^\theta(\f)-\e\int\f\,\mu-O(\e^2J)\right)\\
& = \en_\om(\f)+\en_\om^\theta(\f)-V_\om^\theta\en_\om(\f)-\e\int\f\,\mu-O(\e^2 J). 
\end{align*} 
Injecting this in~\eqref{equ:Jbelow} and using~\eqref{equ:ennab}, we get, as desired, 
$$
\jj_{\om+\theta}(\mu)\ge\jj_\om(\mu)+\jj_\om^\theta(\mu)-O(\e^2 J). 
$$
\end{proof}

\begin{proof}[Proof of Theorem~\ref{thm:diffen}] Assume first $\theta\in\cZ_\om$. Set $J:=\tJ_\om(\mu)$. By Lemma~\ref{lem:diffen} we then have 
$$
\jj_{\om+t\theta}(\mu)\ge\jj_\om(\mu)+t\jj_\om^\theta(\mu)-O(t^2\|\theta\|_\om J). 
$$
For $|t|\ll 1$, we also have $\d(\om+t\theta,\om)\le 1$, and hence $\tJ_{\om+t\theta}(\mu)\lesssim J$, by Corollary~\ref{cor:lipen}. Reversing the roles of $\om$ and $\om+t\theta$, Lemma~\ref{lem:diffen} thus yields
$$
\jj_\om(\mu)\ge\jj_{\om+t\theta}(\mu)-t\jj_{\om+t\theta}^\theta(\mu)-O(t^2\|\theta\|_\om J). 
$$
By Theorem~\ref{thm:twistedhold}, $\jj_{\om+t\theta}^\theta(\mu)\to\jj_\om^\theta(\mu)$ as $t\to 0$, and we conclude, as desired,  
$$
\jj_{\om+t\theta}(\mu)=\jj_\om(\mu+t\jj_\om^\theta(\mu)+o(t). 
$$
as $t\to 0$. 

In the general case, pick $\tau\in\cD_\om$ such that $\theta\in\cZ_{\om_\tau}$ (see Proposition~\ref{prop:normfin}). Then~\eqref{equ:Jmestrans2} yields 
$$
\jj_{\om+t\theta}(\mu)=\jj_{\om_\tau+t\theta}(\mu)-\int\tau\,\mu+\en_{\om+t\theta}(\tau),
$$
see Remark~\ref{rmk:Jnonpos}. We infer
$$
\frac{d}{dt}\bigg|_{t=0}\jj_{\om+t\theta}(\mu)=\jj_{\om_\tau}^\theta(\mu)+\nabla_\theta\en_\om(\tau)=\jj_{\om_\tau}^\theta(\mu)+\jj_\om^\theta(\mu_{\om_\tau})= \jj_\om^\theta(\mu),
$$
where the first equality follows from the first step of the proof and~\eqref{equ:nablaE}, the second one from~\eqref{equ:nablaJ}, and the third from~\eqref{equ:Jomtau}. 
\end{proof}

%
%
%
\section{Coercivity thresholds and free energy}\label{sec:freeen}
As in~\S\ref{sec:twisteddiff}, we assume that the orthogonality and submean value properties hold. We then establish a general continuity result for coercivity thresholds, and apply this to the free energy, which induces the Mabuchi K-energy on potentials.

%
%
\subsection{Continuity of coercivity thresholds}\label{sec:coer}
Fix for the moment $\om\in\cZ_+$ with $[\om]\in\Pos(X)$, and consider an arbitrary functional $F\colon\cM^1\to\R\cup\{+\infty\}$. 

\begin{defi} We define the \emph{coercivity threshold} of $F$ as
\begin{equation}\label{equ:thresh}
\sigma_\om(F):=\sup\left\{\sigma\in\R\mid F\ge\sigma\jj_\om+A\text{ for some }A\in\R\right\}\in[-\infty,+\infty]. 
\end{equation}
We say that $F$ is \emph{coercive} if $\sigma_\om(F)>0$, \ie $F\ge\sigma\jj_\om+A$ for some $\sigma>0$ and $A\in\R$. 
\end{defi}

By Corollary~\ref{cor:lipen}~(ii), the condition that $F$ is coercive (resp.~$\sigma_\om(F)=\pm\infty$) is independent of $\om$.

Recall from~\eqref{equ:JmuMA} that the energy of any $\mu\in\cM^1$ coincides with the quasi-distance to the base point $\mu_\om$, \ie $\jj_\om(\mu)=\d_\om(\mu,\mu_\om)$. As a result, the coercivity threshold measures the linear growth of $F$ with respect to quasi-metric $\d_\om$. As we next show, any other base point can be used in place of $\mu_\om$ (something that could fail for a general quasi-metric). 

\begin{lem}\label{lem:coerind} For any $F\colon\cM^1\to\R\cup\{+\infty\}$ and $\nu\in\cM^1$ we have 
\begin{equation}\label{equ:thresh2}
\sigma_\om(F)=\sup\left\{\sigma\in\R\mid F\ge\sigma\d_\om(\cdot,\nu)+A\text{ for some }A\in\R\right\}. 
\end{equation}
In particular, $\sigma_\om(F)$ only depends on $[\om]\in\Pos(X)$. 
\end{lem}
\begin{proof} For any $\mu\in\cM^1$, \eqref{equ:holdJmes} yields 
$$
\left|\d_\om(\mu,\nu)-\jj_\om(\mu)\right|\lesssim\jj_\om(\nu)^\a\max\{\jj_\om(\mu),\jj_\om(\nu)\}^{1-\a}.
$$
For any $\e>0$, we can thus find $C_\e>0$ (depending on $\nu$) such that 
$$
(1-\e)\jj_\om(\mu)-C_\e\le\d_\om(\mu,\nu)\le(1+\e)\jj_\om(\mu)+C_\e
$$
for all $\mu\in\cM^1$. This implies~\eqref{equ:thresh2}, and the last point follows, since the Dirichlet quasi-metric $\d_\om$ only depends on $[\om]$ (see Theorem~\ref{thm:M1}). 
\end{proof}

\begin{defi} For each $\theta\in\cZ$, we introduce the \emph{twisted coercivity threshold} 
\begin{equation}\label{equ:twistedcoer}
\sigma_\om^\theta(F):=\sigma_\om(F+\jj_\om^\theta).
\end{equation}
\end{defi}
We refer to~\S\ref{sec:Mab} below for a discussion of the concrete case we have in mind. 

\begin{lem}\label{lem:coercohom} The following holds:
\begin{itemize}
\item[(i)] $\sigma_\om^0(F)=\sigma_\om(F)$, and $\sigma_\om^{\theta+t\om}(F)=\sigma_\om^\theta(F)+t$ for all $\theta\in\cZ$ and $t\in\R$; 
\item[(ii)] we have $\sigma_\om(F)\in\R$ (resp.~$\sigma_\om(F)=\pm\infty$) iff $\sigma_\om^\theta(F)\in\R$ (resp.~$\sigma_\om^\theta(F)=\pm\infty$) for all $\theta\in\cZ$. 
\end{itemize}
\end{lem}

\begin{proof} The first point is a direct consequence of~\eqref{equ:Jomom}. For any $\theta\in\cZ$, \eqref{equ:twistedJbd2} yields a constant $C>0$ such that $|\jj_\om^\theta|\le C(\jj_\om+1)$. This implies
$$
\sigma_\om(F)-C\le\sigma_\om^\theta(F)\le\sigma_\om(F)+C,
$$
and (iii) follows. 
\end{proof}

We can now state the main result of this section. 

\begin{thm}\label{thm:threshcont} For any functional $F\colon\cM^1\to\R\cup\{+\infty\}$, the twisted coercivity threshold $\sigma_\om^\theta(F)$ is a continuous function of $([\om],[\theta])\in\Pos(X)\times\HBC(X)$.
\end{thm}
Continuity is understood with respect to the finest vector space topology of $\HBC(X)$, \ie for $[\om]$ and $[\theta]$ constrained in any given finite dimensional subspace. 

\begin{lem}\label{lem:threshcont} There exists $\d_n>0$ only depending on $n$ such that, for all $\om,\om'\in\cZ_+$ with $\d:=\dT(\om,\om')\le\d_n$ and all $\theta,\theta'\in\cZ_\om$, we have 
\begin{equation}\label{equ:threshvar} 
\sigma_{\om'}^{\theta'}(F)\ge(1+O(\d^\a))\left[\sigma_\om^\theta(F)+O(\d^\a\|\theta\|_\om+\|\theta-\theta'\|_\om)\right].
\end{equation}
\end{lem}
\begin{proof} Since $\tJ_\om=\jj_\om+T_\om$, we can replace $\jj_\om$ with $\tJ_\om$ in~\eqref{equ:thresh}, and hence 
\begin{equation}\label{equ:threshbis}
\sigma_\om^\theta(F)=\sup\{\sigma\in\R\mid F+\jj_\om^\theta\ge\sigma\tJ_\om+A\text{ for some }A\in\R\}.
\end{equation}
Pick $\sigma,A\in\R$ such that $F+\jj_\om^\theta\ge\sigma\tJ_\om+A$ on $\cM^1$. By~\eqref{equ:twistedlip} we get
\begin{align*}
F+\jj_{\om'}^{\theta'} & \ge F+\jj_\om^\theta+O(\d^\a\|\theta\|_\om+\|\theta-\theta'\|_\om)\tJ_{\om} \\
& \ge\sigma\tJ_\om+A+O(\d^\a\|\theta\|_\om+\|\theta-\theta'\|_\om)\tJ_{\om}\\
& \ge(1+O(\d^\a))\left[\sigma+O(\d^\a\|\theta\|_\om+\|\theta-\theta'\|_\om\right]\tJ_{\om'}+A, 
\end{align*}
using~\eqref{equ:tJvar}, and hence $\sigma_{\om'}^{\theta'}(F)\ge(1+O(\d^\a))\left[\sigma+O(\d^\a\|\theta\|_\om+\|\theta-\theta'\|_\om\right]$. Choosing $\d_n>0$ such that $1+O(\d^\a)\ge 0$ for $\d\le\d_n$ and taking the supremum over $\sigma$ yields~\eqref{equ:threshvar}. 
\end{proof}

\begin{proof}[Proof of Theorem~\ref{thm:threshcont}] We first show that $\sigma_\om^\theta(F)$ only depends on the classes $[\om]\in\Pos(X)$ and $[\theta]\in\HBC(X)$. For each $\tau\in\cD_\om$, we have 
$$
\sigma_{\om_\tau}^\theta(F)=\sigma_{\om_\tau}(F+\jj_{\om_\tau}^\theta)=\sigma_{\om_\tau}(F+\jj_\om^\theta)=\sigma_\om(F+\jj_\om^\theta)=\sigma_\om^\theta(F), 
$$
where the second equality follows from~\eqref{equ:Jomtau}, and the third from Lemma~\ref{lem:coerind}. This proves that $\sigma_\om^\theta(F)$ only depends on $[\om]$. On the other hand, for any $\p\in\cD$, \eqref{equ:Jomddc} yields 
$$
F+\jj_\om^\theta-C\le F+\jj_\om^{\theta+\ddc\p}\le F+\jj_\om^\theta+C
$$ 
with $C:=2\sup|\p|$. This implies $\sigma_\om^\theta(F)=\sigma_\om^{\theta+\ddc\p}(F)$, which thus only depends on $[\theta]$. 

As noted above, Corollary~\ref{cor:lipen}~(ii) shows that the condition $\sigma_\om(F)=\pm\infty$ is independent of $\om$, and in that case the result trivially holds, since all twisted thresholds are then equal to $\pm\infty$, by Lemma~\ref{lem:coercohom}~(ii). 

Assume now that the twisted thresholds are finite valued. Pick $\om_0\in\cZ_+$ with $[\om_0]\in\Pos(X)$, $\theta_0\in\cZ$. Choose a finite dimensional subspace $W\subset\HBC(X)$ containing $[\om_0]$, $[\theta_0]$, and a finite dimensional subspace $V\subset\cZ$ containing $\theta_0$, and whose image in $\HBC(X)$ contains $W$. Since $\sigma_{\om_0}^{\theta_0}(F)$ only depends on $[\om_0]$, we can assume wlog $V\subset\cZ_{\om_0}$ (see Proposition~\ref{prop:normfin}). After enlarging $V$, we may further assume that it contains $\om_0$. For each $\om,\theta\in V\subset\cZ_{\om_0}$ with $\om$ close enough to $\om_0$, \eqref{equ:threshbis} yields
$$
|\sigma_{\om}^{\theta}(F)-\sigma_{\om_0}^{\theta_0}(F)|\lesssim\d^\a\left[\sigma_{\om_0}^{\theta_0}(F)+O(\d^\a\|\theta_0\|_{\om_0}+\|\theta-\theta_0\|_{\om_0})\right]
$$
with $\d=\dT(\om,\om_0)$. By Lemma~\ref{lem:commdT}, this shows that the restriction of $(\om,\theta)\mapsto\sigma_{\om}^{\theta}(F)$ to $V$ is continuous at $(\om_0,\theta_0)$, and the result follows, since the image of $V$ contains $W$.
\end{proof}

%
%
\subsection{Free energy vs.~Mabuchi K-energy}\label{sec:Mab}
%
%

\subsubsection{The K\"ahler case} Consider first a compact connected K\"ahler manifold $X$. Any smooth metric $\rho$ on $K_X$ induces a volume form $\mu_\rho$, and hence a \emph{(relative) entropy} functional 
\begin{equation}\label{equ:EntK}
\Ent_\rho\colon\cM\to\R\cup\{+\infty\},
\end{equation}
such that 
$$
\Ent_\rho(\mu):=\int\log\left(\frac{\mu}{\mu_\rho}\right)\mu
$$
if $\mu$ is absolutely continuous with respect to $\mu_\rho$, and $\Ent_\rho(\mu)=+\infty$ otherwise. Note that
\begin{equation}\label{equ:enttrans}
\Ent_{\rho'}(\mu)=\Ent_\rho(\mu)+\int(\rho-\rho')\,\mu
\end{equation}
for any other metric $\rho'$ on $K_X$.  As is well-known, the relative entropy can also be written as the Legendre transform
$$
\Ent_\rho(\mu):=\sup_{f\in\cD}\left(\int f\,\mu-\log\int e^f \mu_\rho\right)
$$
which shows that~\eqref{equ:EntK} is convex and lsc. In particular, the restriction $\Ent_\rho\colon\cM^1\to\R\cup\{+\infty\}$ is lsc in the strong topology. While it is not continuous, we nevertheless have (see~\cite[Theorem~4.7]{BDL17}): 

\begin{lem}\label{lem:approxent} For each $\mu\in\cM^1$, there exists a sequence $(\f_i)$ in $\cD_\om$ such that $\mu_i:=\MA_\om(\f_i)$ satisfies $\mu_i\to\mu$ strongly in $\cM^1$ and $\Ent_\rho(\mu_i)\to\Ent_\rho(\mu)$. 
\end{lem}

Extending the terminology of~\cite{thermo}, we introduce:
\begin{defi} We define the \emph{free energy} $\effe_\om\colon\cM^1\to\R\cup\{+\infty\}$ by setting
\begin{equation}\label{equ:free}
\effe_\om(\mu):=\Ent_\rho(\mu)-\Ent_\rho(\mu_\om)+\jj_\om^{\theta_\rho}(\mu),
\end{equation}
where $\theta_\rho\in\cZ$ denotes the curvature of $\rho$. 
\end{defi}
As the notation suggests, $\effe_\om$ is independent of the choice of $\rho$; this follows indeed from~\eqref{equ:enttrans} combined with~\eqref{equ:Jomddc}. Furthermore, \eqref{equ:Jomzero} and~\eqref{equ:Jomtau} show that $\effe_\om$ only depends on the K\"ahler class $[\om]$, up to an additive constant uniquely determined by the normalization $\effe_\om(\mu_\om)=0$. 

The \emph{raison d'\^etre} of the free energy is that its composition with the Monge--Amp\`ere operator coincides with the \emph{Mabuchi K-energy} $\mab_\om\colon\cD_\om\to\R$, \ie we have 
\begin{equation}\label{equ:freemab}
\effe_\om(\MA_\om(\f))=\mab_\om(\f)
\end{equation}
for all $\f\in\cD_\om$. Indeed, in view of~\eqref{equ:nablaJ}, \eqref{equ:freemab} is equivalent to the well-known \emph{Chen--Tian formula} for the K-energy, which can be written as
\begin{equation}\label{equ:CT}
\mab_\om(\f)=\Ent_\rho(\MA_\om(\f))+\nabla_{\theta_\rho}\en_\om(\f)+c
\end{equation}
in the present formalism, for a constant $c\in\R$ determined by the normalization $\mab_\om^\theta(0)=0$, \ie $c=-\Ent_\rho(\mu_\om)$. 
\begin{defi} Define the \emph{coercivity threshold} of $(X,\om)$ as $\sigma(X,\om):=\sigma_\om(\effe_\om)$. 
\end{defi}

\begin{thm} The coercivity threshold of $(X,\om)$ is a continuous function of the K\"ahler class $[\om]$, and it satisfies
$$
\sigma(X,\om)=\sup\left\{\sigma\in\R\mid\mab_\om\ge\sigma\jj_\om+A\text{ on }\cD_\om\text{ for some }A\in\R\right\}. 
$$ 
\end{thm}
By~\cite{CC}, $\sigma(X,\om)>0$ iff there exists a unique constant scalar curvature K\"ahler (cscK) metric in $[\om]$. In particular, we recover the fact, originally proved in~\cite{LS}, that the set of K\"ahler classes of $X$ that contain a unique cscK metric is open.

\begin{proof} Pick a smooth metric $\rho$ on $K_X$. By~\eqref{equ:free}, we have $\sigma(X,\om)=\sigma_\om^{\theta_\rho}(\Ent_\rho)$, and the first point thus follows from Theorem~\ref{thm:threshcont}. As to the second point, it is a simple consequence of~\eqref{equ:freemab} and Lemma~\ref{lem:approxent}.
\end{proof}

\begin{rmk} Given $\theta\in\cZ$, one can more generally consider the \emph{$\theta$-twisted free energy} $\effe_\om^\theta:=\effe_\om+\jj_\om^\theta$, whose composition with the Monge--Amp\`ere operator coincides with the $\theta$-twisted K-energy $\mab_\om^\theta=\mab_\om+\nabla_\theta\en_\om$ (see~\cite{BDL17,CC}). Again, Theorem~\ref{thm:threshcont} shows that the coercivity threshold 
$$
\sigma^\theta(X,\om):=\sigma_\om(\effe_\om^\theta)=\sigma_\om^\theta(\Ent_\rho)
$$
is a continuous function of $[\om]$ and $[\theta]$, while Lemma~\ref{lem:approxent} shows that 
$$
\sigma^\theta(X,\om):=\sup\left\{\sigma\in\R\mid\mab_\om^\theta\ge\sigma\jj_\om+A\text{ on }\cD_\om\text{ for some }A\in\R\right\}
$$
as considered for instance in~\cite{SD2}.
\end{rmk}

%
%
\subsubsection{The non-Archimedean case}\label{sec:coerNA} Consider now a smooth, irreducible projective Berkovich space $X$ over a non-Archimedean field $k$ of characteristic $0$. Pick a PL metric $\rho$ on $K_X$, and define the associated \emph{non-Archimedean entropy} $\Ent_\rho\colon\cM\to\R\cup\{+\infty\}$ by setting 
$$
\Ent_\rho(\mu):=\int(\ld_X-\rho)\,\mu. 
$$
Here $\ld_X$ denotes the Temkin metric (see~\cite{TemkinMetric} and also~\cite{Ste19}) on $K_X$, and $\ld_X-\rho\colon X\to\R\cup\{+\infty\}$ is the corresponding lsc function, using additive notation for metrics (see~\cite[\S 5.7]{konsoib}). Again, $\Ent_\rho$ is convex and lsc, and~\eqref{equ:enttrans} holds for any other choice of PL metric $\rho'$ on $K_X$. See also~\cite{Ino22} for a related notion in the trivially valued case. 

As above, one can then define the (non-Archimedean) \emph{free energy} $\effe_\om\colon\cM^1\to\R\cup\{+\infty\}$ by~\eqref{equ:free}, with $\theta_\rho\in\cZ$ the curvature of $\rho$. Its composition with the Monge--Amp\`ere operator coincides with the (non-Archimedean) \emph{Mabuchi K-energy} $\mab_\om\colon\cD_\om\to\R$, defined by~\eqref{equ:CT}. 

However, the major difference in the non-Archimedean case is that the analogue of Lemma~\ref{lem:approxent} is only a conjecture (compare~\cite[Conjecture~2.5]{trivvalold} and~\cite[Conjecture~4.4]{Li22}). Explicitly: 

\begin{conj}\label{conj:approxent} For each $\mu\in\cM^1$, there exists a sequence of $\om$-psh PL functions $(\f_i)$ such that $\mu_i:=\MA_\om(\f_i)$ satisfies $\mu_i\to\mu$ in $\cM^1$ and $\Ent_\rho(\mu_i)\to\Ent_\rho(\mu)$. 
\end{conj} 

As a consequence, the coercivity threshold $\sigma(X,\om):=\sigma_\om(\effe_\om)$ only satisfies
\begin{equation}\label{equ:threshineq}
\sigma(X,\om)\le\sup\{\sigma\in\R\mid\mab_\om\ge\sigma\jj_\om+A\text{ on }\cD_\om\text{ for some }A\in\R\},
\end{equation}
and equality holds if Conjecture~\ref{conj:approxent} is valid. 

\medskip

These definitions are especially relevant when $k$ is trivially valued and $\om$ lies in $\Amp(X)\hto\cZ_+$ (see~\eqref{equ:numtriv}). Indeed, the free energy $\effe_\om(\mu)$ then coincides with the invariant $\b_\om(\mu)$ introduced and studied in~\cite{nakstab2}; see also~\cite{DL,LiuYa}. By homogeneity with respect to the action of $\R_{>0}$, $\sigma(X,\om)=\sigma_\om(\effe_\om)$ is further equal to the \emph{divisorial stability threshold} $\sigma_\div(X,\om)$, which is positive iff $(X,\om)$ is \emph{divisorially stable}. 

On the other hand, the right-hand side of~\eqref{equ:threshineq} coincides, by definition, with the \emph{K-stability threshold} $\sigma_K(X,\om)$, which is positive iff $(X,\om)$ is \emph{uniformly K-stable} in the sense of~\cite{Der,BHJ1}). Thus divisorial stability implies uniform K-stability, and the converse holds if Conjecture~\ref{conj:approxent} is satisfied. 

\begin{rmk} Assume that $(X,L)$ is a polarized smooth projective variety over $\C$, and pick a K\"ahler form $\om\in c_1(L)$. Using~\cite{SD} and~\cite{Li22}, one can then show that the above thresholds satisfy 
$$
\sigma_\div(X,L)\le\sigma(X,\om)\le\sigma_K(X,L). 
$$
Conjecture~\ref{conj:approxent} would further yield $\sigma_K(X,L)=\sigma_\div(X,L)$, and hence conclude the proof of the `uniform version' of the Yau--Tian--Donaldson conjecture,  as noted in~\cite{Li22,nakstab2}. 
\end{rmk}

%
%
%
%
%
\appendix
\section{Convexity estimates}\label{sec:CS}
%
%
%
%
%
We consider the following data:
\begin{itemize}

\item a surjective map $\pi\colon V\to\Theta$ of $\R$-vector spaces, with fibers $V_\theta=\pi^{-1}(\theta)$; 
	
\item a homogeneous polynomial $F\colon V\to\R$ of degree $n+1$, $n\ge 1$, with associated symmetric multilinear map 
$$
V^{n+1}\to\R\quad(x_0,\dots,x_n)\mapsto x_0\inter x_n,
$$
\ie $F(x)=x^{n+1}$; 

\item a convex cone $P\subset V$ such that 
\begin{equation}\label{equ:posdef}
x^2\cdot x_2\inter x_n\ge 0\text{ for all }x\in V_0\text{ and }x_i\in P. 
\end{equation}
\end{itemize}
\begin{exam}
  The main example that we have in mind is the negative of the energy pairing in~\S\ref{sec:enpair}, where $V=\cZ\times\cD$ with its projection to $\Theta=\cZ$, and $P=\{(\theta,\f)\in V\mid \theta_\f\ge 0\}$. Another example is given by the negative of intersection pairing on a flat projective scheme over $\Z$.
\end{exam}

Our goal is to use~\eqref{equ:posdef} and the resulting \emph{Cauchy-Schwartz} inequality
\begin{equation}\label{equ:CS}
  (x\cdot y\cdot x_2\inter x_n)^2\le(x^2\cdot x_2\inter x_n)(y^2\cdot x_2\inter x_n)
\end{equation}
for all $x,y\in V_0$ and $x_i\in P$ to derive various inequalities and convexity statements.

For each $x\in V$ we define the linear form $F'(x)\in V^\vee$ by
$$
\langle F'(x),y\rangle:=\frac{d}{dt}\bigg|_{t=0}F(x+ty)=(n+1) x^n\cdot y,  
$$
and we set for all $x,y\in V$
\begin{equation}\label{equ:delta}
\d(x,y):=F(x)-F(y)-\langle F'(y),x-y\rangle. 
\end{equation}
A simple computation yields
\begin{equation}\label{equ:dapp}
\d(x,y)=\sum_{j=0}^{n-1}(j+1)(x-y)^2\cdot y^j\cdot x^{n-1-j}.
\end{equation}
In what follows, we fix $\theta\in\Theta$ and set $P_\theta:=V_\theta\cap P$.
\begin{lem}\label{lem:deltaconv}
  We have $\d(x,y)\ge0$ for $x,y\in P_\theta$. Moreover, $F$ is convex on $P_\theta$, and for every $y\in P_\theta$ we have that $x\mapsto\d(x,y)$ is convex. 
\end{lem}
\begin{proof}
  By~\eqref{equ:posdef} and~\eqref{equ:dapp}, we have $\d(x,y)\ge 0$ for $x,y\in P_\theta$, and this implies that $F$ is convex on $P_\theta$; it then follows from~\eqref{equ:delta} that $x\mapsto\d(x,y)$ is convex.
\end{proof}

\begin{thm}\label{thm:CS} For all $x,y,z\in P_\theta$ and $t\in [0,1]$, the following holds: 
\begin{itemize}
\item \emph{quasi-symmetry}:
$$
\d(x,y)\approx\d(y,x); 
$$
\item \emph{quasi-triangle inequality}: 
$$
\d(x,z)\lesssim \d(x,y)+\d(y,z); 
$$ 
\item \emph{quadratic estimate}:
  $$\d(x,(1-t)x+t y)\lesssim t^2\d(x,y);$$
\item \emph{uniform convexity}: 
$$
[(1-t) F(x)+t F(y)]-F((1-t)x+t y)\gtrsim t (1-t) \d(x,y).
$$
\end{itemize}
For any base point $x_*\in P_\theta$ and $x_i,y_i,z_j\in P_\theta$, the following \emph{H\"older estimates} further hold: 
\begin{equation}\label{equ:hold1}
|(x_0-y_0)\cdot(x_1-y_1)\cdot z_2\inter z_n|\lesssim \d(x_0,y_0)^{\a} \d(x_1,y_1)^\a M^{1-2\a}; 
\end{equation}
\begin{equation}\label{equ:hold2}
\left|\langle F'(x_0)-F'(y_0), x_1-y_1\rangle\right|\lesssim \d(x_0,y_0)^{1/2} \d(x_1,y_1)^\a M^{1/2-\a};
\end{equation}
\begin{equation}\label{equ:hold3}
\left|\d(x_0,x_1)- \d(y_0,y_1)\right|\lesssim \max\{\d(x_0,y_0),\d(x_1,y_1)\}^\a M^{1-\a};
\end{equation}
with $\a:=2^{-n}\in(0,1/2]$ and $M=\max_\xi\d(\xi,x_*)$, where in each case $\xi$ ranges over the elements of $P_\theta$ appearing in the left-hand side of the inequality.
\end{thm}
The strategy to get these types of H\"older estimates goes back to~\cite{BBGZ,BBEGZ,trivval}, building on an original idea of~\cite{Blo}. In the rest of this section we prove Theorem~\ref{thm:CS}, largely following~\cite[\S3.3]{trivval}.

\medskip
Given $x,y\in P_\theta$, we set
\[
  d(x,y):=\max_{0\le j\le n-1}(x-y)^2\cdot y^j\cdot x^{n-1-j}. 
\]
Using~\eqref{equ:dapp} it is then clear that
\begin{equation}\label{equ:distances}
  \delta(x,y)\approx \d(y,x)\approx d(x,y)\approx(x-y)^2\cdot(\tfrac12(x+y))^{n-1}.
\end{equation}

To prove the inequality $\d(x,(1-t)x+ty)\lesssim t^2\d(x,y)$ it suffices to prove the corresponding inequality $d(x,(1-t)x+ty)\lesssim t^2d(x,y)$. But
\[
  d(x,(1-t)x+ty)
  =t^2\max_{0\le j\le n-1}(x-y)^2\cdot((1-t)x+ty)^j\cdot x^{n-1-j}
\lesssim t^2d(x,y)
\]
using multilinearity and the binomial theorem.

\medskip
Next we prove what is essentially a special case of~\eqref{equ:hold1}.
\begin{lem}\label{lem:hold1}
  If $x,y,z\in P_\theta$, then
  \[
    (x-y)^2\cdot z^{n-1}\lesssim d(x,y)^{2\a}\max\{d(x,z),d(y,z)\}^{1-2\a}.
  \]
\end{lem}
\begin{proof}
  Set $w:=\frac12(x+y)$, $A:=d(x,y)$, $B:=\max\{d(x,z),d(y,z)\}$, and, for $0\le j\le n-1$,
  \[
    b_j:=(x-y)^2\cdot z^j\cdot w^{n-1-j}.
  \]
  Then $b_0\approx A$, and our goal is to show $b_{n-1}\lesssim A^{2\a}B^{1-2\a}$. By the triangle inequality for the seminorm $v\mapsto\sqrt{v^2\cdot z^{n-1}}$ on $V_0$, we have
  \[
    b_{n-1}\le(\sqrt{(x-z)^2\cdot z^{n-1}}+\sqrt{(y-z)^2\cdot z^{n-1}})^2\le 4B.
  \]
  If $A\ge B$, then $b_{n-1}\le 4B\le 4A^{2\a}B^{1-2\a}$ and we are done, so we may assume $A\le B$. In this case, we show by induction that
  \[
    b_j\lesssim A^{2^{-j}}B^{1-2^{-j}}
  \]
  for $0\le j\le n-1$. The case $j=0$ is clear, so suppose $0\le j\le n-2$, and note that
  \begin{align*}
    b_{j+1}-b_j
    &=(x-y)^2\cdot(z-w)\cdot z^j\cdot w^{n-2-j}\\
    &=(x-y)\cdot(z-w)\cdot x\cdot z^j\cdot w^{n-2-j}
      -(x-y)\cdot(z-w)\cdot y\cdot z^j\cdot w^{n-2-j}.
  \end{align*}
  Here we can use the Cauchy-Schwartz inequality to estimate the last two terms. For example:
  \begin{multline*}
    |(x-y)\cdot(z-w)\cdot x\cdot z^j\cdot w^{n-2-j}|^2\\
    \le((x-y)^2\cdot x\cdot z^j\cdot w^{n-2-j})
    ((z-w)^2\cdot x\cdot z^j\cdot w^{n-2-j}).
  \end{multline*}
  Using that $2w-x=y\in P$, we can bound the first factor by $2b_j$, and the second factor by $2(z-w)^2\cdot z^j\cdot w^{n-1-j}\le 2d(z,w)$. Adding the two terms, we get
  $b_{j+1}-b_j\le 4\sqrt{b_j}\sqrt{d(z,w)}$.
  Now $d(z,w)\approx\d(w,z)\le\max\{\d(x,z),d(y,z)\}\approx B$ using the convexity of $\delta(\cdot,z)$, see Lemma~\ref{lem:deltaconv}. All in all, this yields
  \[
    b_{j+1}-b_j\lesssim\sqrt{b_jB}.
  \]
  for $0\le j\le n-2$. Using the induction hypothesis $b_j\lesssim A^{2^{-j}}B^{1-2^{-j}}$, we get
  \[
    b_{j+1}
    \lesssim b_j+\sqrt{Bb_j}
    \lesssim A^{2^{-j}}B^{1-2^{-j}}+A^{2^{-j-1}}B^{1-2^{-j-1}}\le 2 A^{2^{-j-1}}B^{1-2^{-j-1}},
    \]
where the last inequality follows from our assumption that $A\le B$. We are done.
\end{proof}

Using Lemma~\ref{lem:hold1}, we can now prove the quasi-triangle inequality for $\delta$, or equivalently for $d$. Fix $x,y,z\in P_\theta$, and set $w:=\frac12(x+y)$. Then
\begin{multline*}
  d(x,y)
  \approx(x-y)^2\cdot w^{n-1}
  \lesssim
  (x-z)^2\cdot w^{n-1}+(y-z)^2\cdot w^{n-1}\\
  \lesssim\max\{d(x,z),d(y,z)\})^{2\a}\max\{d(x,w),d(y,w),d(z,w)\}^{1-2\a},
\end{multline*}
by the triangle inequality for the norm $v\mapsto\sqrt{v^2\cdot w^{n-1}}$ and by Lemma~\ref{lem:hold1}. As noted above, the convexity of $\d(\cdot,z)\approx d(\cdot,z)$ gives $d(z,w)\lesssim\max\{d(x,z),d(y,z)\}$, as well as $d(x,w),d(y,w)\lesssim d(x,y)$. Thus
\[
  d(x,y)
  \lesssim\max\{d(x,z),d(y,z)\})^{2\a}\max\{d(x,y),d(x,z),d(y,z)\}^{1-2\a},
\]
which easily implies $d(x,y)\le \max\{d(x,z),d(y,z)\}$, as desired.

\medskip
Next we prove~\eqref{equ:hold1} in general. By the Cauchy--Schwartz inequality we may assume $x_0=x_1$ and $y_0=y_1$. We may also assume $n\ge2$, or else we are done by~\eqref{equ:distances}. Set $z:=\tfrac1{n-1}(z_2+\dots+z_n)$. Then 
\begin{multline*}
  (x_1-y_1)^2\cdot z_2\inter z_n
  \lesssim(x_1-y_1)^2\cdot z^{n-1}
  \lesssim d(x_1,y_1)^{2\a}\max\{d(x_1,z),d(y_1,z)\}^{1-2\a}\\
  \approx \d(x_1,y_1)^{2\a}\max\{\d(x_1,z),\d(y_1,z)\}^{1-2\a}.
\end{multline*}
By the quasi-triangle inequality we have $\d(x_1,z)\lesssim\max\{\d(x_1,x_*),\d(z,x_*)\}$, and by quasi-symmetry and Lemma~\ref{lem:deltaconv} we have $\d(z,x_*)\lesssim\max_{i\ge2}\d(z_i,x_*)$. A similar estimate for $\d(y_1,z)$ completes the proof of~\eqref{equ:hold1}.

\medskip
Next we prove~\eqref{equ:hold2}, which is equivalent to
\[
  |(x_0^n-y_0^n)\cdot(x_1-y_1)|\le\d(x_0,y_0)^{1/2}\d(x_1,y_1)^\a M^{1/2-\a}.
\]
By the Cauchy--Schwartz inequality, we have
\begin{multline*}
  |(x_0^n-y_0^n)\cdot(x_1-y_1)|^2
  =|(x_0-y_0)(x_1-y_1)\sum_{j=0}^{n-1}x_0^jy_0^{n-1-j}|^2\\
  \le\left((x_0-y_0)^2\sum_{j=0}^{n-1}x_0^jy_0^{n-1-j}\right)
  \left((x_1-y_1)^2\sum_{j=0}^{n-1}x_0^jy_0^{n-1-j}\right)
\end{multline*}
Here the first factor on the right is $\approx\d(x_0,y_0)$, whereas the second factor can be  bounded above using~\eqref{equ:hold1}.

\medskip
It only remains to prove~\eqref{equ:hold3}.
By the quasi-triangle inequality for $\d$, it suffices to consider the case when $x_0=y_0$ or $x_1=y_1$. Now
\begin{align*}
  \d(x_0,x_1)-\d(x_0,y_1)
  &=n(x_1^{n+1}-y_1^{n+1})-(n+1)x_0(x_1^n-y_1^n)\\
  &=(x_1-y_1)\sum_{j,k}(x_1^jy_1^{n-j}-x_0x_1^ky_1^{n-k-1}).
\end{align*}
If $j\le k$, then
\[
  x_1^jy_1^{n-j}-x_0x_1^ky_1^{n-k-1}
  =x_1^jy_1^{n-k-1}(y_1^{k-j+1}-x_0^{k-j+1})
  +x_1^jy_1^{n-k-1}x_0(x_0^{k-j}-x_1^{k-j}),
  \]
  and by factoring each term of the right-hand side we see from~\eqref{equ:hold1} that
  \[
    |(x_1-y_1)(x_1^jy_1^{n-k}-x_0x_1^jy_1^{n-k-1})|\lesssim\d(x_1,y_1)^\a M^{1-\a}.
  \]
  The case when $j>k$ is handled in a similar way, and adding all the terms yields
  $|\d(x_0,x_1)-\d(x_0,y_1)|\lesssim\d(x_1,y_1)^\a M^{1-\a}$.

A similar argument shows that $|\d(x_0,x_1)-\d(y_0,x_1)|\lesssim\d(x_0,y_0)^\a M^{1-\a}$, and completes the proof.

%
%
%
%
%
%
\section{Regularization and orthogonality on K\"ahler spaces}\label{sec:orthocomplex}
By relying on a variant of the classical Richberg regularization technique, it was proved in~\cite{BK} that any $\om$-psh function on a compact K\"ahler manifold $(X,\om)$ can be written as the limit of a decreasing sequence of smooth $\om$-psh functions. It is natural to hope that this holds in the singular case as well:

\begin{conj}\label{conj:reg} Let $(X,\om)$ be a compact K\"ahler space. Then any $\om$-psh function $\f$ on $X$ can be written as the pointwise limit of a decreasing sequence $(\f_i)$ of smooth $\om$-psh functions. 
\end{conj}
Note that the conclusion only depends on the K\"ahler class $[\om]$. Besides the case $X$ is nonsingular, mentioned above, we have: 

\begin{exam}\label{exam:CGZ} Conjecture~\ref{conj:reg} holds if $X$ is projective and $[\om]$ lies in the ample cone, \ie the open convex cone generated by classes of ample line bundles on $X$. This is a consequence of~\cite[Theorem~1.1]{CGZ}. 
\end{exam}

\begin{lem}\label{lem:orthocomplex} Let $(X,\om)$ be a compact K\"ahler space for which Conjecture~\ref{conj:reg} holds. Then $[\om]$ has the orthogonality property (cf.~Definition~\ref{defi:ortho}). 
\end{lem}

\begin{proof} Pick a resolution of singularities $\pi\colon Y\to X$, and set $\theta:=\pi^\star\om\ge 0$. Since $\int\theta^n=\int\om^n>0$, the $(1,1)$-class $[\theta]$ is semipositive and big. For any $g\in\Cz(Y)$, consider the $\theta$-psh envelope
$$
\env_\theta(g):=\sup\{\p\in\PSH(Y,\theta)\mid\p\le g\}.
$$
As is well-known, $\env_\theta(\g)$ is $\theta$-psh, and $\MA_\theta(\env_\theta(g))$ is supported in $\{\env_\theta(g)=g\}$, \ie 
\begin{equation}\label{equ:orthobig}
\int_Y(g-\env_\theta(g))\MA_\theta(\env_\theta(g))=0.
\end{equation}
For any $f\in\Cz(X)$, we next claim that $\env_\theta(\pi^\star f)$ is the limit in $\cE^1(Y,\theta)$ of the increasing net $\{\pi^\star\f\}_{\f\in\cD_{\om,<f}}$. Assume this for the moment. By continuity of Monge--Amp\`ere integrals along increasing nets in $\cE^1(Y,\theta)$, and using 
$$
\int_X(f-\f)\MA_\om(\f)=\int_Y(\pi^\star f-\pi^\star\f)\MA_\theta(\pi^\star \f),
$$
for any $\f\in\cD_\om$, we infer
\begin{equation}\label{equ:limint}
\lim_{\f\in\cD_{\om,<f}}\int_X(f-\f)\MA_\om(\f)=\int_Y\left(\pi^\star f-\env_\theta(\pi^\star f)\right)\MA_\theta(\env_\theta(\pi^\star f))=0, 
\end{equation}
by~\eqref{equ:orthobig}. 

To prove the claim, note first that Conjecture~\ref{conj:reg} implies that the increasing net $\{\f\}_{\f\in\cD_{\om,<f}}$ converges pointwise to
$$
\env_\om(f):=\sup\{\p\in\PSH(X,\om)\mid\p\le f\}.
$$
Indeed, given $\d>0$ and a function $\psi\in\PSH(X,\om)$ with $\psi\le f$, Conjecture~\ref{conj:reg} and a Dini-type argument guarantees the existence of a function $\psi'\in\cD_{\om,<f}$ with $\psi\le\psi'< f+\d$. The claim is thus equivalent to the statement that $\pi^\star\env_\om(f)$ coincides a.e.\ with $\f:=\env_{\theta}(\pi^\star f)$. To prove this, pick $\tau\in\PSH(X,\om)$ with $\tau\le f$ and $\{\tau=-\infty\}=X_{\sing}$ (compare proof of Lemma~\ref{lem:submean}). Since $\p_\e:=(1-\e)\f+\e\pi^\star\tau$ is $\pi^\star\om$-psh outside $\pi^{-1}(X_{\sing})$, and $\p_\e\equiv-\infty$ on the latter, we have $\p_\e=\pi^\star\f_\e$ for a unique $\f_\e\in\PSH(X,\om)$ (see~\cite[Th\'eor\`eme~1.10]{DemSMF}). Since $\pi^\star\f_\e\le\pi^\star f$, and hence $\f_\e\le f$, we have $\f_\e\le\env_\om(f)$. Thus 
$$
(1-\e)\f+\e\pi^\star\tau=\pi^\star\f_\e\le\pi^\star\env_\om(f)\le\f,
$$
and hence $\pi^\star\env_\om(f)=\env_\theta(\pi^\star f)$ on $Y\setminus\pi^{-1}(X_{\sing})$. 
\end{proof}
%
%
%
%
%

\end{document}